\documentclass[onefignum,onetabnum]{siamonline220329}

\usepackage{lineno}
\usepackage{multirow}
\usepackage{multicol}
\usepackage{amsfonts}
\usepackage{comment}
\usepackage{amsmath}
\usepackage{amssymb}
\usepackage{color}
\usepackage{url}
\usepackage{wrapfig}

\DeclareMathAlphabet\mathbfcal{OMS}{cmsy}{b}{n}
\newcommand{\ten}[1]{\mathbfcal{#1}} 
\newcommand{\mat}[1]{\mathbf{#1}}
\newcommand{\tmat}[1]{\mathbf{\tilde{#1}}}
\newcommand{\gmat}[1]{\boldsymbol{#1}}

\newcommand{\re}{\mathbb{R}}

\newcommand{\cA}{\ten{A}}
\newcommand{\cI}{\mathcal{I}}
\newcommand{\bE}{\mathbb{E}}
\def\rank{\mbox{rank}}

\newcommand{\sign}{\text{sign}}
\newcommand{\group}{\text{group}}
\newcommand{\tr}{\text{tr}}

\newcommand{\bTheta}{\boldsymbol{\Theta}}
\newcommand{\tbTheta}{\boldsymbol{\tilde{\Theta}}}

\usepackage{braket,amsfonts}

\usepackage{array}

\usepackage[caption=false]{subfig}
\usepackage{pgfplots}

\newsiamthm{claim}{Claim}
\newsiamremark{remark}{Remark}
\newsiamremark{hypothesis}{Hypothesis}
\crefname{hypothesis}{Hypothesis}{Hypotheses}

\usepackage{algorithmic}

\usepackage{graphicx,epstopdf}

\Crefname{ALC@unique}{Line}{Lines}

\usepackage{amsopn}

\usepackage{xspace}
\usepackage{bold-extra}
\usepackage[most]{tcolorbox}

\colorlet{texcscolor}{blue!50!black}
\colorlet{texemcolor}{red!70!black}
\colorlet{texpreamble}{red!70!black}
\colorlet{codebackground}{black!25!white!25}

\lstdefinestyle{siamlatex}{%
  style=tcblatex,
  texcsstyle=*\color{texcscolor},
  texcsstyle=[2]\color{texemcolor},
  keywordstyle=[2]\color{texemcolor},
  moretexcs={cref,Cref,maketitle,mathcal,text,headers,email,url},
}

\tcbset{%
  colframe=black!75!white!75,
  coltitle=white,
  colback=codebackground, 
  colbacklower=white, 
  fonttitle=\bfseries,
  arc=0pt,outer arc=0pt,
  top=1pt,bottom=1pt,left=1mm,right=1mm,middle=1mm,boxsep=1mm,
  leftrule=0.3mm,rightrule=0.3mm,toprule=0.3mm,bottomrule=0.3mm,
  listing options={style=siamlatex}
}

\newtcblisting[use counter=example]{example}[2][]{%
  title={Example~\thetcbcounter: #2},#1}

\newtcbinputlisting[use counter=example]{\examplefile}[3][]{%
  title={Example~\thetcbcounter: #2},listing file={#3},#1}

\DeclareTotalTCBox{\code}{ v O{} }
{ 
  fontupper=\ttfamily\color{black},
  nobeforeafter,
  tcbox raise base,
  colback=codebackground,colframe=white,
  top=0pt,bottom=0pt,left=0mm,right=0mm,
  leftrule=0pt,rightrule=0pt,toprule=0mm,bottomrule=0mm,
  boxsep=0.5mm,
  #2}{#1}

\newsiamthm{assumption}{Assumption}

\begin{tcbverbatimwrite}{tmp_\jobname_header.tex}
\title{Hardware-Efficient Mixed-Precision CP Tensor Decomposition \thanks{The authors are with Department of Electrical and Computer Engineering, University of California at Santa Barbara, CA.
  (Emails: \email{ziy@ucsb.edu}, \email{junnanshan@ucsb.edu}, \email{zzhang01@ucsb.edu}).}}

\author{Zi Yang, Junnan Shan, and Zheng Zhang}

\end{tcbverbatimwrite}
\title{Hardware-Efficient Mixed-Precision CP Tensor Decomposition \thanks{The authors are with Department of Electrical and Computer Engineering, University of California at Santa Barbara, CA.
  (Emails: \email{ziy@ucsb.edu}, \email{junnanshan@ucsb.edu}, \email{zzhang01@ucsb.edu}).}}

\author{Zi Yang, Junnan Shan, and Zheng Zhang}

\ifpdf
\hypersetup{ pdftitle={Guide to Using  SIAM'S \LaTeX\ Style} }
\fi

\begin{document}
\maketitle
\begin{abstract}
    Tensor decomposition has been widely used in machine learning and high-volume data analysis. However, large-scale tensor factorization often consumes huge memory and computing cost. Meanwhile, modernized computing hardware such as tensor processing units (TPU) and Tensor Core GPU has opened a new window of hardware-efficient computing via mixed- or low-precision arithmetic representations. In this paper, we exploit the low-precision representation of tensor factorization, and propose a mixed-precision block stochastic gradient descent (SGD) method to reduce the costs of CP tensor decomposition. Our method achieves robust and fast convergence via a two-stage optimization, i.e., SignSGD followed by mixed-precision SGD. Detailed theoretical analysis is provided to prove the convergence of the proposed mixed-precision algorithm. Numerical experiments on both synthetic and realistic tensor data sets show the superior efficiency of our mixed-precision algorithm compared to full-precision CP decomposition. This work can remarkably reduce the memory, computing and energy cost on resource-constraint edge computing devices. We demonstrate this benefit via an FPGA prototype.
\end{abstract}

\begin{keywords}
Tensor decomposition, mixed-precision optimization, stochastic gradients, hardware-aware algorithms.
\end{keywords}
\begin{MSCcodes}
15A69, 90C15
\end{MSCcodes}

\section{Introduction}
\label{sec:introdution}

As a higher-order generalization of matrices, tensors~\cite{kolda2009tensor} have been used to represent and process multi-dimensional arrays in many science and engineering fields, including quantum physics \cite{dressler2021separability,huggins2019towards,nie2020hermitian,orus2019tensor}, scientific computing~\cite{bigoni2016spectral,richter2021solving}, uncertainty quantification~\cite{dolgov2015polynomial,zhang2016big,zhang2014enabling}, machine learning \cite{anandkumar2014tensor,guo2022learning,hawkins2022towards,hawkins2021bayesian,nie2018complete,novikov2015tensorizing,kim2015compression,sidiropoulos2017tensor} and many others.  Many successful applications rely on efficient tensor decompositions~\cite{bro1997parafac,de2000multilinear,grasedyck2010hierarchical,oseledets2011tensor}, which represent an original high-order high-volume data array with some low-rank factors to achieve huge memory and computing cost reduction. For instance, tensor decomposition has achieved orders-of-magnitude parameter reduction of deep neural networks \cite{hawkins2022towards,hawkins2021bayesian,kim2015compression,novikov2015tensorizing}, enabling their energy-efficient training and deployment on edge devices. As one of the most popular tensor decomposition methods, the CANDECOMP/PARAFAC (CP) decomposition~\cite{bro1997parafac} factorizes a large tensor into the summation of some rank-1 tensors. A CP factorization is often obtained via  algebraic methods \cite{domanov2017canonical,nie2017generating} or numerical optimization techniques such as gradient-based optimization~\cite{ge2017optimization} and alternating minimization \cite{comon2009tensor}. The former provides excellent theoretical guarantees, but are neither noise-resistant nor scalable to high tensor ranks. The later has better efficiency, but computing the full gradients is expensive for high-volume tensor data sets. Motivated by the success in large-scale machine learning, recent approaches use stochastic gradient descent (SGD) methods \cite{battaglino2018practical,beutel2014flexifact,kolda2020stochastic,vervliet2015randomized} to relief the high computation cost in tensor factorization. So far, most (if not all) tensor decomposition algorithms are developed for classical computing platforms (e.g., CPU and conventional GPU) that use double-precision 64-bit or single-precision 32-bit floating-point data representations. 

On the other hand, the recent revolution of artificial intelligence has triggered massive interests in computing hardware that supports mixed-precision and low-precision computation. For instance, Google’s Tensor Processing Units (TPUs)~\cite{jouppi2017datacenter} can easily handle machine learning tasks with 16-bit floating point representations. NVIDIA's tensor Core GPU supports double-, single- and half-precision floating-point operations, as well as various low-precision integer operations. Reconfigurable computing platforms such as field-programmable gate arrays (FPGA) can support arbitrarily low-precision computation to save energy and hardware utilization. These mixed-precision computing platforms are very suitable for the training and inference of deep learning models~\cite{de2018high,de2017understanding,hubara2017quantized,sun2020ultra,zhang2021fpga}, due to their error-resilient activation functions or output operators. Interestingly, recently mixed-precision computing has also shown great success in many scientific computing tasks~\cite{abdelfattah2021survey,buttari2007mixed,carson2018accelerating,carson2020three,carson2022mixed,haidar2020mixed,olivares2010accelerating} such as LU factorization, Cholesky factorization, least square optimization, GMRES. However, mixed-precision computing has been rarely investigated for tensor computation. We envision that similar memory and runtime benefit can be obtained by developing mixed-precision tensor computation algorithms. As the development of 5G and future 6G networks, more and more (possibly private and sensitive) data needs to be processed on resource-constraint edge devices, where mixed-precision tensor computation will play an increasingly important role.


In this paper, we make the first step of exploring low-precision tensor computation by proposing a novel mixed-precision CP tensor decomposition algorithm. By utilizing low-precision stochastic gradient computation in a two-stage optimization framework, our method can remarkably reduce the computation and energy costs of CP decomposition. Our main contributions are summarized below.
\begin{itemize}
\item We propose a computationally efficient CP decomposition via a mixed-precision SGD method. We improve the convergence via a mixed-precision SignSGD initialization. We carefully design the low-precision stochastic gradient computation to maximize the computational efficiency and minimize the accuracy drop via analyzing the sensitivity of each step with respect to the quantization errors. 
   \item  We prove the convergence of the proposed mixed-precision CP decomposition. Under some conditions, we firstly show that the CP decomposition problem is locally strongly convex after proper normalization. Then, we prove that SignSGD with mixed-precision gradients converges to a stationary point up to a noise level. Finally, we prove that the mixed-precision SGD has a locally linear convergence rate for our problem.
   \item Numerical experiments demonstrate that our mixed-precision approach can remarkably reduce the computation cost while attaining similar accuracy to the full-precision algorithm. An FPGA prototype further demonstrates the saving of run-time, hardware resources, and energy on edge computing devices.

\end{itemize}
 We remark that the proposed mixed-precision stochastic gradient can be applied to all SGD-based algorithms for CP tensor decomposition.

\section{Preliminary}
\subsection*{Notation} Throughout the paper, lower-case letters (e.g., $a$) denote scalars; lower-case bold letters (e.g., $\mat{a}$) denote vectors; upper case bold letters (e.g., $\mat{A}$) denote matrices. We use $\mat{1}$ or $\mat{0}$ to denote a vector/matrix whose entries are all $1$ or $0$, respectively. $\mat{I}_n$ is an $n$-by-$n$ identity matrix. We use upper-case calligraphic bold letters (e.g., $\ten{A}$) to denote tensors, which are high-dimensional generalizations of matrices. We use $[n]$ to denote the set of integers $\{1,2,\ldots,n\}$.  For a vector $\mat{v}$, $\|\mat{v}\|$ and $\|\mat{v}\|_1$ denote its Euclidean norm and $1$-norm, respectively. For a matrix $\mat{A}$, $\mat{A}^T$ denotes its transpose; $\tr(\mat{A})$ denotes the trace of $\mat{A}$; $\|\mat{A}\|$ represents the Frobenius norm, and the spectrum norm $\|\mat{A}\|_2$ is the largest singular value of $\mat{A}$. We use MATLAB-style indexing to denote submatrices. For instance, $\mat{A}(i_1:i_2,j_1:j_2)$ denotes the submatrix consisting of the rows from $i_1$ to $ i_2$ and the columns from $j_1$ to $j_2$. The function $\text{sign}(a)$ obtains the sign of $a$, i.e.,
\[
    \text{sign}(a):=\left \{
        \begin{array}{ll}
             1& \text{if } a>0  \\
             0& \text{if } a=0 \\
             -1& \text{if } a<0
        \end{array}.
    \right.
\]
The $\text{sign}$ function can be used for matrices and tensors by applying the function element-wisely. 

For a twice-differentiable function $f:\re^n\to \re$, we use $\nabla f \in \re^n$ and $\nabla^2 f \in \re^{n\times n}$ to denote the gradient and the Hessian matrix of $f$, respectively. The function $f$ is $\lambda$-strongly convex for $\lambda>0$ if the smallest eigenvalue of $\nabla^2 f$ is not less than $\lambda$. Equivalently, $f$ is $\lambda-$strongly convex if 
\[
    f(\mat{y}) \ge f(\mat{x}) + \nabla f(\mat{x})^T(\mat{y}-\mat{x}) + \frac{\lambda}{2}\|\mat{y}-\mat{x}\|^2,\, \forall \mat{x},\mat{y} \in \re^n.
\]
For the vector valued function $\mat{h}(\mat{x})=(h_1(\mat{x}),\ldots,h_m(\mat{x}))$ where $h_i:\re^n \to \re$, the Jacobian matrix $\mat{J}_{\mat{h}}(\mat{x})$ is 
\[
    \mat{J}_{\mat{h}}(\mat{x}):=\big(\nabla h_1(\mat{x}),\ldots,\nabla h_m(\mat{x}) \big)^T.
\]



\subsection{Tensors}
Tensors can be regarded as multi-dimensional data arrays \cite{landsberg2012tensors}. The space of real tensors with order $m$ and dimension $N_1,N_2,\ldots,N_m$ is denoted by $\re^{N_1\times N_2 \times \cdots \times N_m }$. The $(i_1, i_2, \cdots, i_m)$-th element of a tensor $\cA \in \re^{N_1\times N_2 \times \cdots \times N_m }$ is denoted as $a_{i_1,\ldots,i_m}$ for $1\le i_j\le N_j$. The mode-$k$ unfolding of $\cA$ is the matrix $\mat{A}_{[k]} \in \re^{N_k\times (\prod\limits_{i=1}^m N_i)/N_k}$ which is obtained by reshaping $\cA$ with the $k$th dimension being the leading dimension. The Frobenius norm of $\cA$ is  
\[
    \|\ten{A}\|_{\rm F}:=\sqrt{\sum_{i_1,\ldots,i_m}^{N_1,\ldots,N_m}a_{i_1,\ldots,i_m}^2}.
\]
For vectors $\{ \mat{u}_i\in \re^{N_i}\} _{i=1}^m$, their outer product forms an order-$m$ rank-1 tensor
\[
\ten{B}=\mat{u}_1\circ \mat{u}_2 \circ \cdots \circ \mat{u}_m \iff b_{i_1,\ldots,i_m} =\prod \limits_{k=1}^m \mat{u}_k (i_k).
\]
A tensor $\ten{A}\in \re^{N_1\times N_2 \times \cdots \times N_m } $ is said to have a rank-$r$ CP decomposition if there exist matrices $\{ \mat{U}_i \in \re^{N_i\times r}\}_{i=1} ^m$ such that
\[
   \ten{A} = [\![\mat{U}_1, \mat{U}_2, \cdots, \mat{U}_m ]\!] := \sum_{j=1}^r \mat{U}_1(:,j)\circ \cdots \circ \mat{U}_m(:,j).
\]
The smallest integer $r$ that ensures the above equality is called the CP rank of $\cA$, denoted by $\rank(\cA)$. The Khatri-Rao product of matrices $\mat{U}_1,\ldots,\mat{U}_m$ is a column-wise Kronecker product, i.e.,
\[
    \mat{U}_1\odot \cdots \odot \mat{U}_m:=\left[
    \otimes_{i=1}^m \mat{U}_i(:,1),\ldots,\otimes_{i=1}^m \mat{U}_i(:,r)
    \right],
\]
where $\otimes$ denotes the Kronecker product. It holds that $\mat{A}_{[k]}=\mat{U}_k (\odot_{i=1,i\neq k}^m \mat{U}_i)^T$.

\subsection{Precision Reprensentations} \label{subsec:low-precision}

\begin{table}[t]
    \centering
    \begin{tabular}{|c|c|c|c|c|c|c|}
    \hline 
    Type & Bits & Sign &  Exponent & Significand & Min & Max\\
    \hline
       \texttt{FP}$_{16}$  & $16$ & $1$ & $5$ & $10$ & $6.1\times 10^{-5} $ & $6.6\times 10^{4}$  \\ 
       \hline
    \texttt{FP}$_{32}$ & $32$ & $1$ & $8$ & $23$ & $1.2\times 10^{-38} $ & $3.4\times 10^{38}$ \\
    \hline
    \texttt{FP}$_{64}$ & $64$ & $1$ & $11$ & $52$ & $2.2\times 10^{-308} $ & $1.8\times 10^{308}$ \\ \hline
    \end{tabular}
    \caption{Floating Point Representations}
    \label{tab:floating}
\end{table}

In practice, numbers are represented and processed as binary strings on digital computing hardware. The binary strings can represent numbers in either fixed-point format or floating-point format. We use \texttt{INT}$_n$ and \texttt{FP}$_n$ to denote an $n$-bit fixed-point format and an $n$-bit floating-point format, respectively. The representation format is directly related to the precision of the represented number. Hence, the representation format is also called precision format.

An \texttt{INT}$_n$ data representation uses $n$ bits, where the first bit stores the sign and the other $n-1$ bits store the absolute value. The set of numbers that the \texttt{INT}$_n$ format can represent is 
$
    \{-2^{n-1},-2^{n-1}+1,\ldots,0,1,\ldots,2^{n-1}-1\}.
$
The \texttt{FP}$_n$ format uses $1$ bit to store the sign, $N_1$ bits to store significand, and $N_2$ bits to store exponent, where $N_1+N_2+1=n$. Then, the number is represented by 
$
    \text{sign} \times \text{significand} \times 2^{\text{exponent}}.
$
\texttt{FP}$_{16}$ (half precision), \texttt{FP}$_{32}$ (single precision), and \texttt{FP}$_{64}$ (double precision) are most commonly used and are supported by most devices. Their bits for each part and representation ranges are described in Table \ref{tab:floating}. Floating-point arithmetic operations are much more expensive than fixed-point arithmetic operations with the same number of bits. Clearly, low-bit representations consume less memory and computation resources but cause larger rounding-off errors. Table \ref{tab:compare_mat} compares the run-time of matrix multiplications under different precision formats on tensor core GPU. The chosen test sizes are common in computing gradients for the proposed Algorithm \ref{alg:SGD} as in \eqref{eq:Qg}. We can see that the \texttt{INT}$_8$ multiplications are $4\times $ to $5\times$ faster than \texttt{FP}$_{32}$ multiplications.
 
\begin{table}[]
    \centering
    \begin{tabular}{|c|c|c|c|}
    \hline
        $(m,k,n)$ & \texttt{INT}$_8$ time & \texttt{FP}$_{16}$ time & \texttt{FP}$_{32}$ time  \\
        \hline
        (240,$240^2$,256) & 232 ($\mathbf{4.47\times}$)	&675 (${1.54\times}$) &	1037 ($1 \times$)\\
        \hline
        (60,$60^3$,64) & 794 ($\mathbf{4.55\times}$) &	2456 (${1.47\times}$) &	 3615 ($1 \times$)\\
        \hline
        (24,$24^4$,32) &1139 ($\mathbf{4.89\times}$) &	3784 (${1.47\times}$)&	5571 ($1 \times$)\\ \hline

    \end{tabular}
    \caption{Time comparisons of matrix multiplications of $m\times k$ and $k\times n$  under various precisions on GPU. The times are measured in microseconds ($\mu$s). }
    \label{tab:compare_mat}
\end{table}


Deterministic rounding and stochastic rounding methods can be used to round a high-precision number to a lower precision. For a given precision format $p$, let ${\cal R}(p)$ be the set of numbers that can be represented by the format $p$. The ceiling and floor functions with precision $p$ are defined as 
\begin{align}
      \lceil y \rceil_p := & \min \{v \in {\cal R}(p) \cup \{+\infty\}| v\ge y \}, \nonumber \\ \lfloor y \rfloor_p := & \max \{v \in {\cal R}(p) \cup \{-\infty\}| v\le y \}. \nonumber
\end{align}
When the precision $p$ is not specified, we use ${\cal R}(p)=\mathbb{N}$ by default.
The quantization function $\texttt{Q}^D_{p,\delta}$, with precision $p$, scaling factor $\delta$, and deterministic rounding, is defined as \[
   \texttt{Q}^D_{p,\delta}(x) = \left \{
      \begin{array}{ll}
         \delta \lceil {x/\delta}\rceil_p & \text{ if } {x/\delta} \geq \frac{1}{2} \left( \lceil {x/\delta}\rceil_p + \lfloor {x/\delta}\rfloor_p \right) \\
         \delta \lfloor {x/\delta}\rfloor_p & \text{ if } {x/\delta}< \frac{1}{2} \left( \lceil {x/\delta}\rceil_p + \lfloor {x/\delta}\rfloor_p \right).
      \end{array}\right. .
\]
The quantization function $\texttt{Q}^S_{p,\delta}$ with stochastic rounding is 
\[
   \texttt{Q}^S_{p,\delta}(x) = \left \{
      \begin{array}{ll}
         \delta \lceil {x/\delta}\rceil_p & \text{with probability } \frac{{x/\delta} - \lfloor {x/\delta}\rfloor_p}{\lceil {x/\delta}\rceil_p- \lfloor{x/\delta}\rfloor_p}  \\
         \delta \lfloor {x/\delta}\rfloor_p & \text{with probability } \frac{\lceil {x/\delta}\rceil_p - {x/\delta}}{\lceil {x/\delta}\rceil_p - \lfloor{x/\delta}\rfloor_p}
      \end{array}\right. .
\]
The stochastic rounding ensures that the quantization is unbiased, i.e.,
$
      \bE(\texttt{Q}^S_{p,\delta}(x)|x) = x.
$

\section{Proposed Algorithm} \label{sec:alg}

This section presents a mixed-precision SGD-type algorithm to reduce the memory and computation cost of CP tensor decomposition. This method has a linear convergence rate when it gets close to the optimal solution.  A mixed-precision SignSGD method is utilized at the beginning to improve the convergence of the whole framework. 


\subsection{Mixed-Precision CP Decomposition} 
Given a tensor $\cA\in \re^{N_1\times \cdots \times N_m}$, the rank-$r$ CP tensor decomposition can be formulated as the optimization problem 
\begin{equation} \label{eq:TD}
   \min_{\boldsymbol{\Theta}}  f(\boldsymbol{\Theta})=\|\ten{A} -  [\![\mat{U}_1, \cdots, \mat{U}_m ]\!]\|_{\rm F}^2, \; {\rm with} \; \boldsymbol{\Theta}=\left\{ \mat{U}_i \in \mathbb{R}^{N_i \times r}\right \}_{i=1}^m.
\end{equation}
This problem can be rewritten as 
\begin{equation} \label{eq:TD sum}
   \min_{\boldsymbol{\Theta}} f:=\frac{1}{N} \sum_{i_1=1}^{N_1}\cdots \sum_{i_m=1}^{N_m}(a_{i_1\ldots i_m} - [\![\mat{U}_1(i_1,:),\cdots, \mat{U}_m(i_m,:)]\!])^2,
\end{equation}
where $N:=N_1\cdots N_m$. Since the cost function is the summation of $N$ functions, we can naturally apply an SGD-type method to solve the optimization.

Instead of using standard SGD \cite{bottou2018optimization}, we present a mixed-precision SGD-type algorithm to solve Problem \eqref{eq:TD}. Let $\mat{U}_1^s,\ldots,\mat{U}_m^s$ be the tensor factor matrices in the $s$-th iteration and $\texttt{Q}(\tilde{g}_i^s)$ be the quantized stochastic gradient with respect to $\mat{U}_i^s$. Corollary \ref{cor:convex} shows that Problem \eqref{eq:TD} is locally strongly convex around the true decomposition if the leading rows of $\mat{U}_i(1,:)$ are fixed for $i=2,\ldots,m$. We propose to update variables as 
\[
    \mat{U}_1^{s+1} = \mat{U}_1^s - \alpha_s \texttt{Q}(\tmat{g}_1^s),
\]
\[
    \mat{U}_i^{s+1}(2:N_i,:) = \mat{U}_i^s(2:N_i,:) - \alpha_s \texttt{Q}(\tmat{g}_i^s)(2:N_i,:).
\]

Problem \eqref{eq:TD} has many stationary points, and the mixed-precision SGD can easily converge to a local optimizer without a good initialization point. We propose to use mixed-precision SignSGD to find a good initialization for SGD, which updates variables as follows:
\[
    \mat{U}_i^{s+1} = \mat{U}_i^s - \alpha_s \text{sign}(\texttt{Q}(\tmat{g}_i^s)).
\]
The mixed-precision SignSGD only uses the sign of the gradient to update parameters. 
Consequently, it is more robust against non-convexity and quantization errors. In practice, we find that SignSGD is unlikely to be trapped by a stationary point. This motivates us to firstly run mixed-precision SignSGD for a number of iterations. When the error becomes small, we switch to mixed-precision SGD for better accuracy and faster convergence. 

The learning rate is updated as $\alpha_{s+1}=\eta \alpha_s$ every $K$ iterations for some constant $1>\gamma>0$. It is the multi-stage update rule. The complete mixed-precision CP decomposition \eqref{eq:TD} is presented in Algorithm \ref{alg:SGD}.

\begin{algorithm}[t]
\caption{Mixed-Precision Stochastic Gradient Algorithm for Tensor Decomposition}
\label{alg:SGD}
\begin{algorithmic}[1]
\STATE{\textbf{Input}:  tensor $\ten{A}\in \re^{N_1\times \cdots \times N_m}$, rank $r$, initialization $\{\mat{U}_i^0 \in \re^{N_i\times r}\}_{i=1}^m$, initial learning rates $\alpha^{\text{sign}}_0,\alpha^{\text{SGD}}_0>0$, thresholds $\epsilon_1>\epsilon_2>0$, quantizations $\texttt{Q}_1,\texttt{Q}_2$, sample sizes $\{n_i\}_{i=1}^m$, learning rate update intervals $K^{\text{sign}},K^{\text{SGD}}$, learning rate update constants $\eta^{\text{sign}},\eta^{\text{SGD}}$.}
\STATE{Let $s=0$.}
\STATE{$\alpha_0=\alpha_0^{\text{sign}}$.}
\WHILE{$\|\ten{A} -  [\![\mat{U}_1^s,\cdots,\mat{U}_m^s]\!]\|/\|\ten{A}\|>\epsilon_1$}
\STATE{Compute the mixed-precision gradient $\{\texttt{Q}(\tilde{\mat{g}}_{i}^s)\}_{i=1}^m$ as in Algorithm \ref{alg:grad}.}
\STATE{$\mat{U}_i^{s+1}= \mat{U}_i^s- \alpha_s \text{sign}(\texttt{Q}(\tilde{\mat{g}}_{i}^s))$.}
\STATE{$\alpha_{s+1}=\left\{\begin{array}{ll}
     \eta^{\text{sign}}\alpha_s,& \text{ if ($(s+1)$ mod $K^{\text{sign}}$) = 0} \\
      \alpha_s,& \text{ otherwise} 
\end{array} \right.$. }
\STATE{$s=s+1.$}
\ENDWHILE
\STATE{Let $\alpha_{s}=\alpha_0^{\text{SGD}}$, $s^{\text{sign}}=s$.}
\WHILE{$\|\ten{A} - [\![\mat{U}_1^s,\cdots,\mat{U}_m^s]\!]\|/\|\ten{A}\|>\epsilon_2$.}
\STATE{Compute the mixed-precision gradient $\{\texttt{Q}(\tilde{\mat{g}}_{i}^s)\}_{i=1}^m$ as in Algorithm \ref{alg:grad}.}
\STATE{$\mat{U}_1^{s+1} = \mat{U}_1^s - \alpha_s \texttt{Q}(\tilde{\mat{g}}_{1}^s)$.}
\STATE{$\mat{U}_i^{s+1}(2:N_i,:) = \mat{U}_i^s(2:N_i,:) - \alpha_s \texttt{Q}(\tilde{\mat{g}}_{i}^s)(2:N_i,:),\, i=2,\ldots,m$ \hfill }
\STATE{$\alpha_{s+1}=\left\{\begin{array}{ll}
      \eta^{\text{SGD}}\alpha_s,& \text{ if ($(s-s^\text{sign}+1)$ mod $K^{\text{SGD}}$) = 0} \\
     \alpha_s,& \text{ otherwise} 
\end{array} \right.$.}
\STATE{$s=s+1$.}
\ENDWHILE
\STATE{\textbf{Output}: factor matrices $\{\mat{U}_i^s\}_{i=1}^m$.}
\end{algorithmic}
\end{algorithm}

\subsection{Mixed-Precision Block Stochastic Gradient} \label{sec:mixed-precision gradient}
Gradient computation is often the most expensive part in SGD-type algorithms. This subsection describes how to efficiently compute the mixed-precision stochastic gradient used in Algorithm \ref{alg:SGD}.

Problem \eqref{eq:TD sum} is well-structured, therefore we use block sampling to maximize the usage of parallel computing. In each iteration, we uniformly sample a subset of indices $\cI_i\subset [N_i]$ with $|\cI_i|=n_i$ for $i=1,\ldots,m$. Then, we consider the cost function 
\[
   f_\cI:=\frac{1}{n}\|\ten{A}(\cI_1,\ldots, \cI_m) -[\![ \mat{U}_1(\cI_1,:),\cdots, \mat{U}_m(\cI_m,:) ]\!]\|_{\rm F}^2,\; {\rm with}\; n=n_1\cdots n_m.
\]
The gradient of $f_\cI$ with respect to $\mat{U}_i(\cI_i,:)$ is 
\begin{equation} \label{eq:bsg}
   \nabla_{\mat{U}_i(\cI_i,:)} f_\cI= -\frac{2}{n}\left(\ten{A}(\cI_1,\ldots, \cI_m) - [\![ \mat{U}_1(\cI_1,:),\cdots, \mat{U}_m(\cI_m,:) ]\!]\right)_{[i]} \odot_{j=1,j\neq i}^m \mat{U}_j(\cI_j,:).
\end{equation}
Therefore, the stochastic gradient $\tmat{g}_{i}:= \nabla_{ \mat{U}_i} f_\cI \in \re^{N_i\times r} $ is given as
\begin{equation} \label{eq:tg}
   \tmat{g}_{i}(j_i,:) =\left \{
   \begin{array}{ll}
     \nabla_{\mat{U}_i(j_i,:)} f_\cI & \text{if $j_i\in \cI_i$} \\
    0 & \text{if $j_i\notin \cI_i$}
   \end{array}\right. .
\end{equation}
We regard $\cI:=(\cI_1,\ldots,\cI_m)$ as a random variable. Each $\cI_i$ is sampled uniformly, hence it holds that 
$ \bE_\cI[\tmat{g}_{i}] = \mat{g}_i$ for $i=1,\ldots,m,$
where $\mat{g}_i:= \nabla_{ \mat{U}_i} f$.

We compute the quantized value of the block stochastic gradient $\tmat{g}_{i}$ \eqref{eq:tg} as follows:

\begin{subequations}\label{eq:Qg}
\begin{eqnarray} \label{eq:Qg-a}
 &\ten{M}&:= -\ten{A}(\cI_1,\ldots, \cI_m) + [\![ \texttt{Q}_{1}(\mat{U}_1(\cI_1,:)),\cdots, \texttt{Q}_{1}(\mat{U}_m(\cI_m,:))]\!] ,\\
    \label{eq:Qg-c}    &\mat{V}_i&:= \odot_{j=1,j\neq i}^m \texttt{Q}_{1}(\mat{U}_j(\cI_j,:)),\\
    \label{eq:Qg-d}    &\texttt{Q}(\tmat{g}_{i})(\cI_i,:)&:= \texttt{Q}_2(\mat{M}_{[i]}) \texttt{Q}_2(\mat{V}_i),
\end{eqnarray}
\end{subequations}
where $\texttt{Q}_1,\texttt{Q}_2$ are two quantization functions as described in Section \ref{subsec:low-precision}. The steps for computing the stochastic gradient in mixed-precision are summarized in Algorithm \ref{alg:grad}.

\begin{algorithm}[t]
\caption{Compute Mixed-Precision Stochastic Gradient}
\label{alg:grad}
\begin{algorithmic}[1]
\STATE{\textbf{Input}: tensor $\ten{A}\in \re^{N_1\times \cdots \times N_m}$, rank $r$, factor matrices $\{\mat{U}_i^0 \in \re^{N_i\times r}\}_{i=1}^m$, quantization functions $\texttt{Q}_1,\texttt{Q}_2$, sample sizes $\{n_i\}_{i=1}^m$.}
\STATE{Randomly sample $\cI_i\subset [n_i]$ with $|\cI_i|=n_i$ for $i\in [m]$.}
\STATE{Compute $\ten{M}= -\ten{A}(\cI_1,\ldots, \cI_m) + [\![ \texttt{Q}_{1}(\mat{U}_1(\cI_1,:)),\cdots, \texttt{Q}_{1}(\mat{U}_m(\cI_m,:))]\!]$.}
\STATE{Compute $\mat{V}_i= \odot_{j=1,j\neq i}^m \texttt{Q}_{1}(\mat{U}_j(\cI_j,:))$ for $i\in [m]$.}
\STATE{Compute $\texttt{Q}(\tmat{g}_{i})(\cI_i,:)= \texttt{Q}_2(\mat{M}_{[i]}) \texttt{Q}_2(\mat{V}_i)$.}
\STATE{\textbf{Output}: Mixed-precision stochastic gradient $\{\texttt{Q}(\tmat{g}_{i})\}_{i=1}^m$.}
\end{algorithmic}
\end{algorithm}

The subtraction in \eqref{eq:Qg-a} and the Khatri–Rao product in \eqref{eq:Qg-c} are both sensitive to quantization errors, and extremely low-precision quantization function $\texttt{Q}_1$ will cause bad convergence behavior. Therefore, we use precision \texttt{FP}$_{16}$ and scale $\delta=1$ for $\texttt{Q}_1$, i.e., $\texttt{Q}_1=\texttt{Q}_{\text{FP}16,1}.$
The last matrix multiplication \eqref{eq:Qg-d} is more robust against errors. Consequently, the quantization functions $\texttt{Q}_2$ can use an extremely low precision. Practically, \texttt{INT}$_4$ and \texttt{INT}$_8$ always work well, and \texttt{INT}$_2$ can work when the tensor rank $r$ is small. For the specific quantization $\texttt{Q}_{\texttt{INT}_b,\delta}(\mat{X})$ for a matrix $\mat{X}$, the scaling factor $\delta$ depends on $\mat{X}$ and the precision \texttt{INT}$_b$. We typically set $\delta$ slightly less than $\frac{\max \{\mat{X} \}}{2^{b-1}-1}$. This ensures most entries of $\mat{X}$ lie in the representation range of \texttt{INT}$_b$ while preserving low quantization errors.



\subsubsection*{Complexity Analysis}
The sub-tensor $\ten{M}$ in \eqref{eq:Qg-a} is only computed once for all $i\in [m]$, and the computation requires around $2nr$ arithmetic operations.  Computing each $\mat{V}_i$ in \eqref{eq:Qg-c} needs $\frac{N}{N_i}r$ arithmetic operations, so the total number of operations of step \eqref{eq:Qg-c} is $\sum_{i=1}^m \frac{n}{n_i}r$. Step \eqref{eq:Qg-d} involves a tensor unfolding along its $i$th dimension. The matrix multiplication \eqref{eq:Qg-d} for each $i$ requires about $2nr$ operations. In total, we will do $m$ such multiplications and the total number of operations is $2mnr$. Therefore, the most expensive step in \eqref{eq:Qg} is the matrix multiplications \eqref{eq:Qg-d}. Fortunately, \eqref{eq:Qg-d} is robust against quantization noises, and its cost can be reduced significantly by using ultra low-precision quantization functions. Suppose that each arithmetic operation of precision $p$ costs $c_p$ computation resources. Computing the mixed-precision gradient as in \eqref{eq:Qg} requires $C(p_1,p_2)=c_{p_1}\left(2nr+\sum_{i=1}^m \frac{n}{n_i}r\right)+2c_{p_2}mnr$ resources, where $p_1,p_2$ are the precision formats used by $\texttt{Q}_1,\texttt{Q}_2$ respectively. In practice, we typically choose $p_1$ as $\texttt{FP}_{16}$ and $p_2$ as some low-bit fixed-point format. The computation resource consumed by a specific representation format is proportional to the number of bits. On modern hardware, fixed-point operations typically use less resources and are much faster than floating-point operations. More specifically, fixed-point operations use less than half resources of floating-point operations with the same number of bits \cite{hettiarachchi2020integer}. Therefore, we have the estimation
$
    c_{\texttt{FP}_{16}}\approx \frac{1}{2}c_{\texttt{FP}_{32}}, \, c_{\texttt{INT}_{b}}\approx \frac{b}{64}c_{\texttt{FP}_{32}}.
$
Then, the estimated costs of \eqref{eq:Qg} under full-precision and low-precision are 
\[
    C(\texttt{FP}_{32},\texttt{FP}_{32}) \approx(2+2m+\sum_{i=1}^m \frac{1}{n_i}) nrc_{\texttt{FP}_{32}}\approx (2+2m) nrc_{\texttt{FP}_{32}},
\]
\[
    C(\texttt{FP}_{16},{\texttt{INT}_{b}}) \approx(1+\frac{b}{32}m+\sum_{i=1}^m \frac{1}{n_i}) nrc_{\texttt{FP}_{32}}\approx (1+\frac{b}{32}m) nrc_{\texttt{FP}_{32}}.
\]
The computation saving of using mixed-precision is 
\[
    C(\texttt{FP}_{16},{\texttt{INT}_{b}})/C(\texttt{FP}_{32},\texttt{FP}_{32}) \approx \frac{1+\frac{b}{32}m}{2+2m}.
\]
The cost reduction of our proposed mixed-precision gradient is more obvious for high-order tensors and smaller number of bits. Table \ref{tab:complexity} shows the normalized computational cost for orders $m=3,4,5$ and precision $\texttt{INT}_{8},\texttt{INT}_{4},\texttt{INT}_{2}$, respectively.

\begin{table}[]
    \centering
    \begin{tabular}{|c|c|c|c|}
    \hline 
         &  $\texttt{INT}_{8}$&$\texttt{INT}_{4}$ & $\texttt{INT}_{2}$\\
         \hline
        $m=3$ & 21.9\% & 17.2\% & 14.8\% \\
        \hline
        $m=4$ & 20.0\% & 15.0\% & 12.5\% \\
        \hline
        $m=5$ & 18.75\% & 13.5\% & 10.9\% \\ \hline
    \end{tabular}
    \caption{Normalized computation cost compared with full-precision for various orders and precisions.}
    \label{tab:complexity}
\end{table}



\section{Convergence Analysis}\label{sec:convergence}
This section presents the convergence result of Algorithm \ref{alg:SGD}.  Under some generic conditions, we prove that the tensor decomposition problem \eqref{eq:TD} is locally strongly convex after proper normalization.  We prove that the mixed-precision SignSGD converges to some stationary points up to some noise caused by stochasticity and quantization errors. We also prove that the mixed-precision SGD has a locally linear convergence rate around the global minimizer.

\subsection{Locally Strong Convexity} \label{subsec:convex}
This subsection shows the locally strong convexity of the problem \eqref{eq:TD} after proper normalization. Note that Problem \eqref{eq:TD} itself is non-convex and it does not have local convexity as well. Suppose that the tensor $\cA$ has the CP decomposition $\ten{A} = [\![\mat{U}_1, \mat{U}_2, \cdots, \mat{U}_m ]\!]$.
Then, it holds that 
\begin{equation}\label{eq:decom scale}
    \cA  = \sum_{j=1}^r c_{1,j}\mat{U}_1(:,j) \circ \cdots \circ c_{m,j} \mat{U}_m(:,j),
\end{equation}
for any $c_{i,j}$'s as long as $\Pi_{i=1}^m c_{1,j}=1$. Therefore, the CP decomposition problem \eqref{eq:TD} has an infinite number of minimizers, but many solutions differ only with scaling factors. Therefore, we fix the elements $\{\mat{U}_i(1,:)\}_{i=2}^m$ and assume that $\mat{U}_i(1,:) = \mat{1}^T$ for $2\le i\le m$ without loss of generality. The CP decomposition problem \eqref{eq:TD} now becomes 
\begin{equation} \label{eq:TD scale}
 \min_{\tbTheta}\tilde{f}(\tbTheta) := \left\|\ten{A} - [\![\mat{U}_1, \big[\mat{1}^T;\tmat{U}_2 \big], \cdots , \big[\mat{1}^T;\tmat{U}_m \big] ]\!]\right \|^2,
\end{equation}
where $\tbTheta:=(\mat{U}_1,\tmat{U}_2,\ldots,\tmat{U}_m)$ and $\mat{U}_1\in \re^{N_1\times r},\tmat{U}_i\in \re^{(N_i-1)\times r}$ for $i=2,\ldots,m$. 

It can be shown that the normalized problem \eqref{eq:TD scale} is strongly convex around its global minimizers. For the tensor $\cA \in \mathbb{R}^{N_1 \times N_2 \times \ldots \times N_m}$ with $N_1\ge N_2\ge \cdots \ge N_m$, we define the largest rank $r_m$ such that the problem \eqref{eq:TD scale} is locally strongly convex. Let 
\begin{equation} \label{eq:r3}
    r_3 := 
         \tilde{N}_{m-2} \lfloor \frac{\tilde{N}_{m-1}\tilde{N}_m}{\tilde{N}_{m-2}+\tilde{N}_{m-1}+\tilde{N}_m-2 }\rfloor,
\end{equation}
where $\tilde{N}_{m-2},\tilde{N}_{m-1}$ and $ \tilde{N}_{m}$ are the largest integers such that (i) $\tilde{N}_{m-2}$ is even, (ii) $\tilde{N}_{m-2}\ge \tilde{N}_{m-1} \ge \tilde{N}_{m}$, and (iii) $N_i \ge \tilde{N}_i$ for $i=m-2,m-1,m$. Then, the upper bound $r_m$ is computed recursively by
\begin{equation} \label{eq:rm}
    r_{k}:=N_{m-k+1} \min\{r_{k-1},\lfloor \frac{N_{m-k+2}\cdots N_m}{N_{m-k+1}+\cdots+N_m-k+1} \rfloor\},\, k=4,\ldots,m.
\end{equation}
The upper bound $r_m$ is around $\frac{N_1\cdots N_m}{N_1+\cdots+N_m-m+1} $ when $N_1,\ldots,N_m$ are large. 

The locally strong convexity holds generically when $r\le r_m$. We say a property is generic if it is true on the whole space except a subset with zero measure \cite{cox2013ideals}. The rigorous result is presented in Theorem \ref{thm:convex}.

\begin{theorem} \label{thm:convex}
    Suppose that $N_1\ge N_2\ge \cdots \ge N_m$ and $r\le r_m$ for $r_m$ in \eqref{eq:rm}. Let  
    \[
        \cA :=[\![ \mat{U}_1^*, \big[\mat{1}^T;\tmat{U}_2^*\big], \cdots, \big[\mat{1}^T;\tmat{U}_m^*\big] ]\!],
    \]
    where $\mat{U}_1^*\in \re^{N_1\times r},\tmat{U}_i^*\in \re^{(N_i-1)\times r},2\le i\le m$.
    Then, for generic $\tbTheta^*:=(\mat{U}_1^*,\ldots,\tmat{U}_m^*)$, the Hessian matrix $\nabla^2 \tilde{f}(\tbTheta^*)$ is positive definite.
    

\end{theorem}
\begin{proof}
    See the proof in the Appendix \ref{sec:proof of convexity}.
\end{proof}

Problem \eqref{eq:TD scale} scales the leading rows to all one vectors, which simplifies the theoretical analysis. In practice, the leading rows can be scaled to arbitrary non-zero vectors. Consider the problem 
\begin{equation} \label{eq:TD scale arbitrary}
 \min_{\tbTheta}\tilde{f}_{A}(\tbTheta) := \|\ten{A} - [\![ \mat{U}_1 , \big[\mat{u}_2^T;\tmat{U}_2\big] , \cdots , \big[\mat{u}_m^T;\tmat{U}_m\big] ]\!]\|_{\rm F}^2,
\end{equation}
where $\mat{u}_i \in \re^{r}$ and $(\mat{u}_i)_j\neq 0,\, \forall j \in [r]$ and $i \in [m]$. Problem \eqref{eq:TD scale} can be converted Problem \eqref{eq:TD scale arbitrary} via some invertible transformations. The invertibility preserves the positive definiteness of the Hessian. Therefore, Problem \eqref{eq:TD scale arbitrary} preserves the locally strong convexity. 

\begin{corollary}\label{cor:convex}
Suppose $N_1\ge N_2\ge \cdots \ge N_m$ and $r\le r_m$ in \eqref{eq:rm}. Let $\mat{u}_2,\ldots,\mat{u}_m$ be vectors in $\re^{r}$ whose elements are all nonzero and 
\[
        \cA := [\![\mat{U}_1^*, \big[\mat{u}_2^T;\tmat{U}_2^*\big] , \cdots , \big[\mat{u}_m^T;\tmat{U}_m^*\big]]\!],
    \]
    where $\mat{U}_1^*\in \re^{N_1\times r},\tmat{U}_i^*\in \re^{(N_i-1)\times r},2\le i\le m$.
    Then, for generic $\tbTheta^*:=(\mat{U}_1^*,\ldots,\tmat{U}_m^*)$, the Hessian $\nabla^2 \tilde{f}_A(\tbTheta^*)$ is positive definite and there exists an open set $O$ containing $\tbTheta^*$ and a constant $\lambda>0$ such that the function $\tilde{f}_A(\tbTheta)$ in \eqref{eq:TD scale arbitrary} is $\lambda$-strongly convex in $O$.

\end{corollary}
\begin{proof}
    To simplify the descriptions, here we regard $\tbTheta$ as a vector including all optimization variables.
    There exists a nonsingular matrix $\mat{D}$ such that $\tilde{f}_A(\tbTheta)= \tilde{f}(\mat{D}\tbTheta)$. The Hessian $\nabla^2 \tilde{f}(\mat{D}\tbTheta^*)$ is positive definite by Theorem \ref{thm:convex}. It holds that 
    \[
        \nabla^2 \tilde{f}_A(\tbTheta^*) = \mat{D}(\nabla^2 \tilde{f}(\mat{D}\tbTheta^*)) \mat{D}^T.
    \]
    Thus, the Hessian
    $
        \nabla^2 \tilde{f}_A(\tbTheta^*) 
    $
    is positive definite. Since the eigenvalues of a matrix are continuous with respect to all matrix elements \cite{horn2012matrix}, there exists a constant $\lambda>0$ and an open set $O$ containing $\tbTheta^*$ such that the smallest eigenvalue of $\nabla^2 \tilde{f}_A(\tbTheta) $ is not less than $\lambda$ in $O$. In other words, $\tilde{f}_A(\tbTheta)$ is $\lambda$-strongly convex in $O$.
\end{proof}

Based on Corollary \ref{cor:convex}, we can prove that Algorithm \ref{alg:SGD} has a local convergence rate in $O$ after switching to mixed-precision SGD.

\subsection{Convergence of Algorithm \ref{alg:SGD}}
We show that (1) the mixed-precision SignSGD in Algorithm \ref{alg:SGD} converges to a stationary point up to some noise, (2) the mixed-precision SGD in Algorithm \ref{alg:SGD} has a linear convergence rate around the true CP decomposition. 

Let $\bTheta^s:=(\mat{U}_1^s,\ldots,\mat{U}_m^s)$ denote the factor matrices at the $s$-th iteration. Suppose that $\{\bTheta^s\}_{s=0}^{S_1}$ and $\{\bTheta^s\}_{s=S_1+1}^{S_2}$ are generated by mixed-precision SignSGD and mixed-precision SGD respectively in Algorithm \ref{alg:SGD}. Recall that $f$ is the objective function defined in \eqref{eq:TD}. We make the following assumptions.

\begin{assumption} \label{assumption:bound H}
Assume that 
$
    \|\nabla^2 f(\bTheta^s)\|_2 \le L,\, s=0,\ldots,S_2.
$

\end{assumption}

\begin{assumption} \label{assumption:bound variance}
Let $\tilde{\mat{g}}^s$ be the stochastic gradient at the $s$-th iteration. Assume that for $s\in [S_2]$ and $i\in [m]$, it holds
\[
\bE(\|\tilde{\mat{g}}_i^s(j,k)-\mat{g}_i^s(j,k)\|^2)\le \sigma_g^2, \, \bE(\|\texttt{Q}(\tilde{\mat{g}}_i^s)(j,k)-\tilde{\mat{g}}_i^s(j,k)\|^2)\le \sigma_Q^2,\, \forall j\in [N_i], k\in [r].
\]

\end{assumption}

Assumption \ref{assumption:bound H} assumes the Hessian matrices are bounded, which is widely used in the convergence analysis of SGD methods. Assumption \ref{assumption:bound variance} ensures the variance of the stochastic gradient and the quantization error are both bounded. Under Assumption \ref{assumption:bound variance}, the quantized stochastic gradient can be bounded as 
 \begin{eqnarray*}
    \bE(\|\texttt{Q}(\tilde{\mat{g}}_i^s)(j,k)-{\mat{g}}_i^s(j,k)\|)&\le& \bE(\|\texttt{Q}(\tilde{\mat{g}}_i^s)(j,k)-{\tmat{g}}_i^s(j,k)\|) + \bE(\|\tilde{\mat{g}}_i^s(j,k)-{\mat{g}}_i^s(j,k)\|)\\
    &\le& \sqrt{\bE(\|\texttt{Q}(\tilde{\mat{g}}_i^s)(j,k)-\tilde{\mat{g}}_i^s(j,k)\|^2)}+\sqrt{\bE(\|\tilde{\mat{g}}_i^s(j,k)-\mat{g}_i^s(j,k)\|^2)}\\
    &\le& \sigma_Q+\sigma_g.
 \end{eqnarray*}

\subsubsection{Convergence of Mixed-Precision SignSGD}
We show the convergence of the mixed-precision SignSGD in Algorithm \ref{alg:SGD}. Our proof is partially motivated by \cite{bernstein2018signsgd}.
\begin{theorem} \label{thm:signSGD}
    Let $\{\bTheta^s\}_{s=0}^{S_1}$ be the sequence generated by the SignSGD update in Algorithm \ref{alg:SGD}. Under Assumption \ref{assumption:bound H} and Assumption \ref{assumption:bound variance}, we have 
    \begin{equation} \label{eq:signSGD}
        \sum_{s=0}^{S_1-1}  \sum_{i=1}^m \frac{n_i}{N_i}\alpha_s\|\mat{g}_i^s\|_1 \le f(\bTheta^0) +\frac{1}{2}r L \sum_{s=0}^{S_1-1}\sum_{i=1}^m n_i\alpha_s^2  + 2r(\sigma_g+\sigma_Q)\sum_{s=1}^{S_1-1}\sum_{i=1}^m n_i\alpha_s.
    \end{equation}
\end{theorem}

\begin{proof}
Under Assumption \ref{assumption:bound H}, it holds that
\begin{equation} \label{eq:signSGDproof}
\begin{aligned}
    f(\bTheta^{s+1}) \le& f(\bTheta^s) + \sum_{i=1}^m \tr\big((\mat{g}_i^s)^T(\mat{U}^{s+1}_i-\mat{U}^{s}_i)\big) + \sum_{i=1}^m \frac{L}{2}\|\mat{U}^{s+1}_i-\mat{U}^{s}_i\|_{\rm F}^2 \\
    =& f(\bTheta^{s}) - \alpha_s\sum_{i=1}^m \tr \big((\mat{g}_i^s(\cI_i^s,:))^T \sign(\texttt{Q}(\tilde{\mat{g}}_i^s)(\cI_i^s,:))\big) + \sum_{i=1}^m \frac{L}{2}\|\alpha_s \sign(\texttt{Q}(\tilde{\mat{g}}_i^s)(\cI_i^s,:))\|^2\\
    =& f(\bTheta^{s}) + \frac{1}{2}\alpha_s^2 r L\sum_{i=1}^m n_i - \alpha_s \sum_{i=1}^m \|\mat{g}_i^s(\cI_i^s,:)\|_1\\
     &\qquad +2\alpha_s \sum_{i=1}^m \sum_{j\in \cI_i^s}\sum_{k=1}^r |\mat{g}_i^s(j,k)|I\left(\sign \left(\mat{g}_i^s\left(j,k\right)\right)\neq \sign\left( \texttt{Q}\left(\tilde{\mat{g}}_i^s\left(j,k\right)\right)\right)\right),
\end{aligned} 
\end{equation}
where $I$ is the indicator function such that $I(\text{true})=1,I(\text{false})=0$.
Considering the part $I(\sign(\mat{g}_i^s(j,k))\neq \sign( \texttt{Q}(\tilde{\mat{g}}_i^s(j,k)))$ in the above, we have 
\begin{eqnarray*}
    \bE[I(\sign(\mat{g}_i^s(j,k))\neq \sign( \texttt{Q}(\tilde{\mat{g}}_i^s)(j,k)] &=& 
    \mathbb{P}[I\big(\sign(\mat{g}_i^s(j,k))\neq \sign( \texttt{Q}(\tilde{\mat{g}}_i^s)(j,k))\big)]\\
    &\le& \mathbb{P}[\|\texttt{Q}(\tilde{\mat{g}}_i^s)(j,k)-\mat{g}_i^s(j,k)\|\ge |\mat{g}_i^s(j,k)| ] \\
    &\le & \frac{\bE[\|\texttt{Q}(\tilde{\mat{g}}_i^s)(j,k)-\mat{g}_i^s(j,k)\| ]}{|\mat{g}_i^s(j,k)|} \\
    &\le& \frac{\sigma_g+\sigma_Q}{|\mat{g}_i^s(j,k)|}
\end{eqnarray*}
It implies that 
\[
    \bE[\sum_{j\in \cI_i^s}\sum_{k=1}^r |\mat{g}_i^s(j,k)|I(\sign(\mat{g}_i^s(j,k))\neq \sign( \texttt{Q}(\tilde{\mat{g}}_i^t)(j,k))]  \le \sum_{j\in \cI_i^s}\sum_{k=1}^r (\sigma_g+\sigma_Q)=n_i r(\sigma_g+\sigma_Q).
\]
Then, we take expectation on both sides of \eqref{eq:signSGDproof} and get
\[
    f(\bTheta^{s+1})\le f(\bTheta^s) + \frac{1}{2}\alpha_s^2 r L\sum_{i=1}^m n_i - \alpha_s \sum_{i=1}^m \frac{n_i}{N_i}\|\mat{g}_i^s\|_1+2\alpha_s r(\sigma_g+\sigma_Q)\sum_{i=1}^m n_i.
\]
After summing up both sides for $s=0,\ldots,S_1-1$ and rearrangement, we have
\[
    \sum_{s=0}^{S_1-1}  \sum_{i=1}^m \frac{n_i}{N_i}\alpha_s\|\mat{g}_i^s\|_1 \le f(\bTheta^0)-f(\bTheta^{S_1}) +\frac{1}{2}r L \sum_{s=0}^{S_1-1}\sum_{i=1}^m n_i\alpha_s^2  + 2r(\sigma_g+\sigma_Q)\sum_{s=0}^{S_1-1}\sum_{i=1}^m n_i\alpha_s.
\]
It implies the result \eqref{eq:signSGD} since $f(\bTheta^{S_1})\ge 0$.
\end{proof}



In practice, we usually choose a relatively large constant learning rate $\alpha_s=\alpha$ to accelerate the convergence at the beginning. We prove in Corollary \ref{cor:signSGD} that a constant learning rate provides $O(\frac{1}{T_1})$ convergence rate up to some noise.
\begin{corollary}\label{cor:signSGD}
Under conditions of Theorem \ref{thm:signSGD}, if $\alpha_s=\alpha$, then 
\[
    \min_{s=0,\ldots,S_1-1} \|\mat{g}_i^s\|_1 \le O(\frac{1}{S_1})+\frac{1}{\gamma}(\frac{\alpha L}{2}+2\sigma_g+2\sigma_Q)r\sum_{i=1}^m n_i,
\]
where $\gamma=\min \limits_{i\in [m]} \frac{n_i}{N_i}$.
\end{corollary}
\begin{proof}

Equation \eqref{eq:signSGD} implies that 
\begin{eqnarray*}
     \min_{s=0,\ldots,S_1-1} \|\mat{g}_i^s\|_1 &\le& \frac{1}{\gamma}\frac{f(\bTheta^0)}{\sum_{s=0}^{S_1-1} \alpha_s} + \frac{1}{\gamma}\left(\frac{1}{2}r L\sum_{i=1}^m n_i  \frac{\sum_{t=0}^{T_1-1}\alpha_s^2}{\sum_{t=0}^{T_1-1} \alpha_s} + 2r(\sigma_g+\sigma_Q)\sum_{i=1}^m n_i\right) \\
    &=& \frac{f(\bTheta^0)}{S_1 \gamma \alpha} + \frac{1}{\gamma}(\frac{\alpha L}{2}+2\sigma_g+2\sigma_Q)r\sum_{i=1}^m n_i\\
    &=& O(\frac{1}{S_1})+\frac{1}{\gamma}(\frac{\alpha L}{2}+2\sigma_g+2\sigma_Q)r\sum_{i=1}^m n_i.
\end{eqnarray*}

\end{proof}

\subsubsection{Convergence of Mixed-Precision SGD}
In this subsection, we show the locally linear convergence rate of the mixed-precision SGD in Algorithm \ref{alg:SGD}. We make the following extra assumption.
\begin{assumption}\label{assumption:descent}
Assume that for some $\theta\in [0,1)$, it holds 
\[
    \|\mathbb{E}(\texttt{Q}(\tmat{g}_i^s)-\tmat{g}_i^s)\| \le \theta \|\mat{g}_i^s\|, \quad i\in [m], \; s=S_1+1,\ldots,S_2.
\]
\end{assumption}
Assumption \ref{assumption:descent} assumes the quantized stochastic gradient is a good descent direction in expectation. 
If the quantization function $Q$ uses independent stochastic rounding, then Assumption \ref{assumption:descent} is true for $\theta=0$ since $\mathbb{E}(\texttt{Q}(\tmat{g}^s_i))=\bE(\tmat{g}^s_i) = \mat{g}^s_i$. Assumption \ref{assumption:descent} still holds for deterministic rounding as long as the quantization error is not large.

Let $\mat{u}_i:=\mat{U}^{S_1}_i(1,:)\in \re^r$ for $i=2,\ldots,m$. Suppose $
        \cA := [\![\mat{U}_1^*, \big[\mat{u}_2^T;\tmat{U}_2^*\big], \cdots \circ \big[\mat{u}_m^T;\tmat{U}_m^*\big] ]\!]$,
    where $\mat{U}_1^*\in \re^{N_1\times r},\tmat{U}_i^*\in \re^{(N_i-1)\times r},2\le i\le m$. The objective function now becomes $\tilde{f}_A(\tbTheta)$ as in \eqref{eq:TD scale arbitrary}, which is locally $\lambda$-strongly convex by Corollary \ref{cor:convex}. Consequently, the convergence result of SGD for strongly convex functions can be applied.

\begin{theorem}\label{thm:SGD}
    Suppose the tensor $\cA$ satisfies the conditions of Corollary \ref{cor:convex}. Let $\{\bTheta^s\}_{s=S_1}^{S_2}$ be the sequence generated by the mixed-precision SGD in Algorithm \ref{alg:SGD} and $\alpha_s=\alpha$ be the learning rate. If $\tbTheta^s$ is in the set $O$ as in Corollary \ref{cor:convex} and $\bE(\|\texttt{Q}(\tilde{\mat{g}}^s)\|)\le G$ for $s=S_1,\ldots,S_2$, then under Assumption \ref{assumption:bound H}, Assumption \ref{assumption:bound variance}, and Assumption \ref{assumption:descent}, it holds that 
    \[
        \bE(\|f(\bTheta^{S_2})\|^2) \le \frac{\alpha L G}{2\lambda (1-\theta)} + (1-\alpha \lambda (1-\theta))^{S_2-S_1}(f(\bTheta^{S_1})-\frac{\alpha L G}{2\lambda (1-\theta)})
    \]
    where $\lambda$ is the strong convexity parameter in Corollary \ref{cor:convex}.
\end{theorem}
\begin{proof}
     Under Assumption \ref{assumption:descent}, the gradient $\tmat{g}_i^s$ satisfies
\begin{eqnarray*}
    \bE((\mat{g}_i^s)^T \texttt{Q}(\tmat{g}_i^s)) &=& (\mat{g}_i^s)^T\bE(\tmat{g}_i^s) + (\mat{g}_i^s)^T\bE(\texttt{Q}(\tmat{g}_i^s-\tmat{g}_i^s)\\
    &\ge&  \|\mat{g}_i^s\|^2 - \|\mat{g}_i^s\| \|\bE(\texttt{Q}(\tmat{g}_i^s-\tmat{g}_i^s) \| \\
    &\ge& (1-\theta)  \|\mat{g}_i^s\|^2.
\end{eqnarray*}
    By Corollary \ref{cor:convex}, the function $\tilde{f}_A(\tbTheta)$ is $\lambda$-strongly convex in $O$ containing $\tbTheta^*$. Algorithm \ref{alg:SGD} is minimizing the function $\tilde{f}_A(\tbTheta)$ after switching to mixed-precision SGD. It also holds that $f(\bTheta^{S_2})=f_A(\tbTheta^{S_2})$. Therefore, the result is a direct conclusion of Theorem 4.6 in \cite{bottou2018optimization}.
\end{proof}

\section{Numerical Experiments}\label{sec:tests}

\subsection{Implementation Details}
Recall that the block stochastic gradient is computed as in Algorithm \ref{alg:grad}. The quantization function $\texttt{Q}_1$ use \texttt{FP}$_{16}$, scale factor $\delta=1$, and deterministic rounding, i.e.,
$
  \texttt{Q}_1 = \texttt{Q}^D_{\text{\texttt{FP}$_{16}$},1}.
$
The quantization function $\texttt{Q}_2$ use \texttt{INT}$_b$ precision and deterministic rounding. When quantizing the matrix $\mat{X}$, we use the scale factor $\delta=\frac{\max |\mat{X}|}{c}$, where $c\ge 2^{b-1}-1$. Specifically, we use $c=10,30,200$ for \texttt{INT}$_2$, \texttt{INT}$_4$, and \texttt{INT}$_8$ respectively. In this section, the precision of Algorithm \ref{alg:SGD} always means the precision of $\texttt{Q}_2$.

Our implementation uses the Python package CuPy \cite{cupy_learningsys2017}. For fair comparisons between different precisions, we implement the matrix multiplication by CUTLASS kernels \cite{Kerr_CUTLASS_2022}. However, due to the lack of support for extremely low-bit fixed-point integer representations in Python, we only compare the running time between \texttt{INT}$_8$ and \texttt{FP}$_{32}$ on GPU. The learning rate for the mixed-precision SignSGD in Algorithm \ref{alg:SGD} is set as $0.5$ initially and is updated as $\alpha = 0.3\alpha$ every 1000 iterations. The mixed-precision SGD stage uses the constant learning rate $\alpha=0.01$. For tensors of orders $3,4,5$, we use the sample sizes $|\cI_i|=0.2N_i,0.3N_i,0.4N_i$ respectively. The size of the sampled sub-tensor is roughly $1\%$ of the original tensor.

Suppose that Algorithm \ref{alg:SGD} outputs the factor matrices $\{\mat{U}_i\}_{i=1}^m$ for the input tensor $\cA$. We use a relative error to measure the qualify of our results, which is defined as 
\[
    \text{error} = \frac{\|\cA - [\![\mat{U}_1,\cdots,\mat{U}_m]\!]\|_{\rm F}}{\|\cA\|_{\rm F}}.
\]

\begin{figure}[t]
     \centering
     \subfloat[Initialization with max $1.0$]{\label{fig:initial_1}\includegraphics[width=0.45\textwidth]{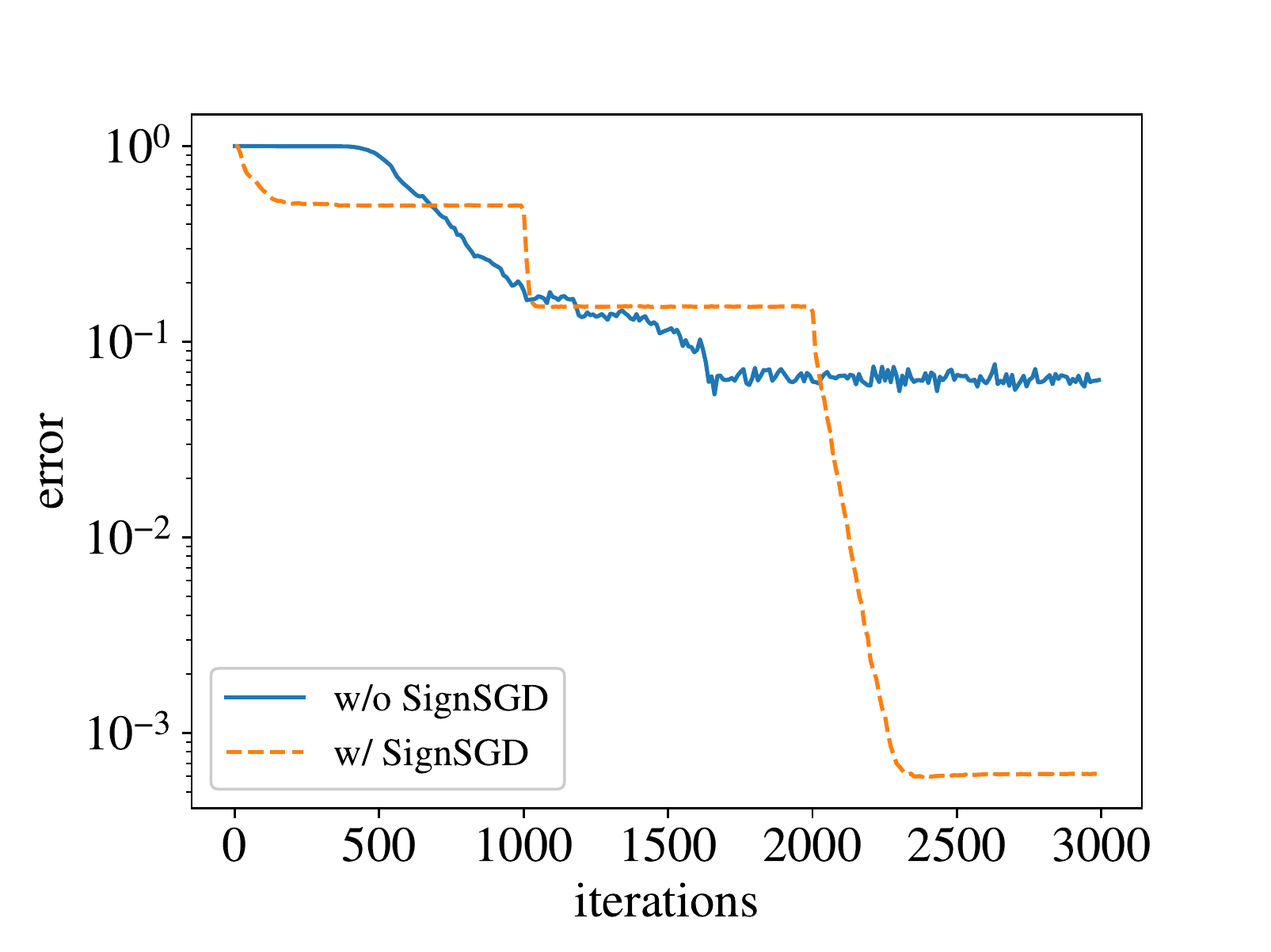}} 
     \subfloat[Initialization with max $0.1$ and $0.01$]{\label{fig:initial_0.1}\includegraphics[width=0.45\textwidth]{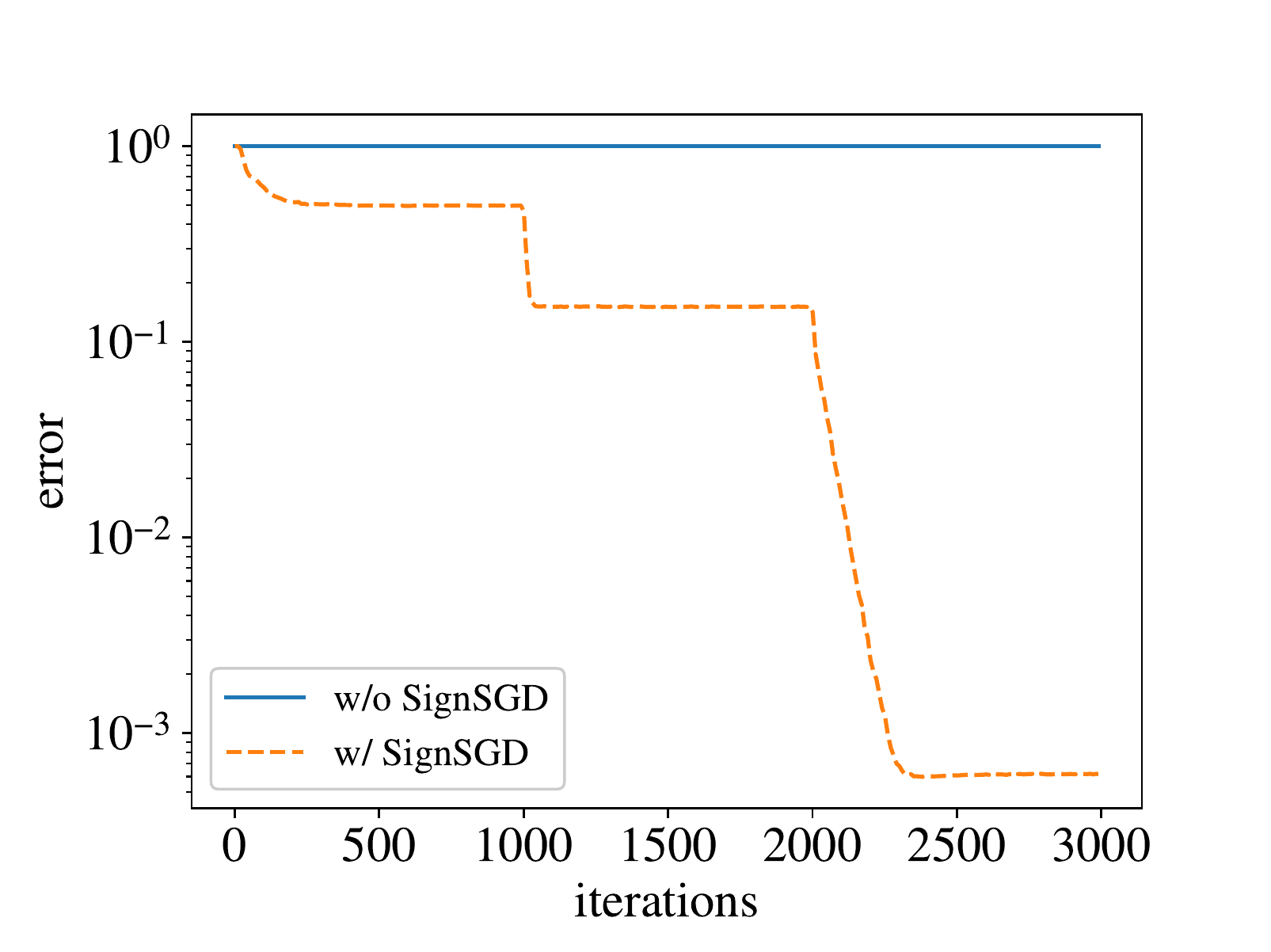}} 
\caption{Performance of Algorithm \ref{alg:SGD} with and without SignSGD initialization.}
\label{figure:initial compare}
\end{figure}

\subsection{Synthetic Examples}
We first test the runtime and convergence of Algorithm \ref{alg:SGD} under various precisions on some synthetic tensor benchmarks. 

\subsubsection{Role of SignSGD Initialization}
This section runs the experiment in full precision to show the influence of SignSGD initialization to the convergence of the whole algorithm. The results with different initialization methods are shown in Figure \ref{figure:initial compare}. The ``max" in Figure \ref{figure:initial compare} is the maximum absolute value of each $\{\mat{U}_i\}_{i=1}^m$. Algorithm \ref{alg:SGD} without SignSGD is trapped by a stationary point and fails to converge with $\max=1.0$ as shown in Figure \ref{fig:initial_1}. After we decrease max to $0.1$ and $0.01$, Algorithm \ref{alg:SGD} without SignSGD stays at zero, which is a stationary point. In contrast, Algorithm \ref{alg:SGD} with SignSGD converges well for max=$1.0,0.1,0.01$. The result demonstrates that the SignSGD initialization can greatly improve the convergence of Algorithm \ref{alg:SGD}.


\subsubsection{Time Comparison in Different Precisions}
We test the runtime of Algorithm \ref{alg:SGD} to reach the same relative error $10^{-3}$ under different precisions. We specifically compare the runtime of Algorithm \ref{alg:SGD} with \texttt{INT}$_8$ and \texttt{FP}$_{32}$ respectively, and the result is summarized in Table \ref{tab:time comparison}. Figure \ref{fig:size 1200}, \ref{fig:size 200}, \ref{fig:size 60} show that the runtime increases linearly as the tensor rank increases for both low-precision and full-precision. The reduction ratio remains the same as the rank changes. We can observe significant time savings when using the \texttt{INT}$_8$ format for all sizes, ranks, and orders. The time saving is also more remarkable as the tensor order increases. This is because $m$ large-size matrix multiplications are computed in low precision in Algorithm \ref{alg:SGD} for tensors with order $m$. Therefore, higher order $m$ brings in more time savings. The detailed complexity analysis is in Section \ref{sec:mixed-precision gradient}. 

\begin{table}[t]
    \centering
    \begin{tabular}{|c|c|c|c|c|c|}
    \hline
        Dimension &  Sample size & Rank & \texttt{INT}$_8$ time (s) & \texttt{FP}$_{32}$ time (s)& Speed up \\
        \hline
        (1200,1200,1200) & (240,240,240)&256 & 17.96 &28.15 & \textbf{1.56}$\times$\\
        \hline
        (200,200,200,200) & (60,60,60,60)&64 & 29.90 &56.44 & \textbf{1.88}$\times$\\
        \hline
        (60,60,60,60,60) & (24,24,24,24,24)& 32 & 41.71 &114.89 & \textbf{2.75}$\times$\\ \hline
    \end{tabular}
    \caption{Time comparison between \texttt{FP}$_{32}$ and \texttt{INT}$_8$ of Algorithm \ref{alg:SGD} for various dimensions}
    \label{tab:time comparison}
\end{table}

\begin{figure}[t]
     \centering
     \subfloat[Size (1200,1200,1200)]{\label{fig:size 1200}\includegraphics[width=0.33\textwidth]{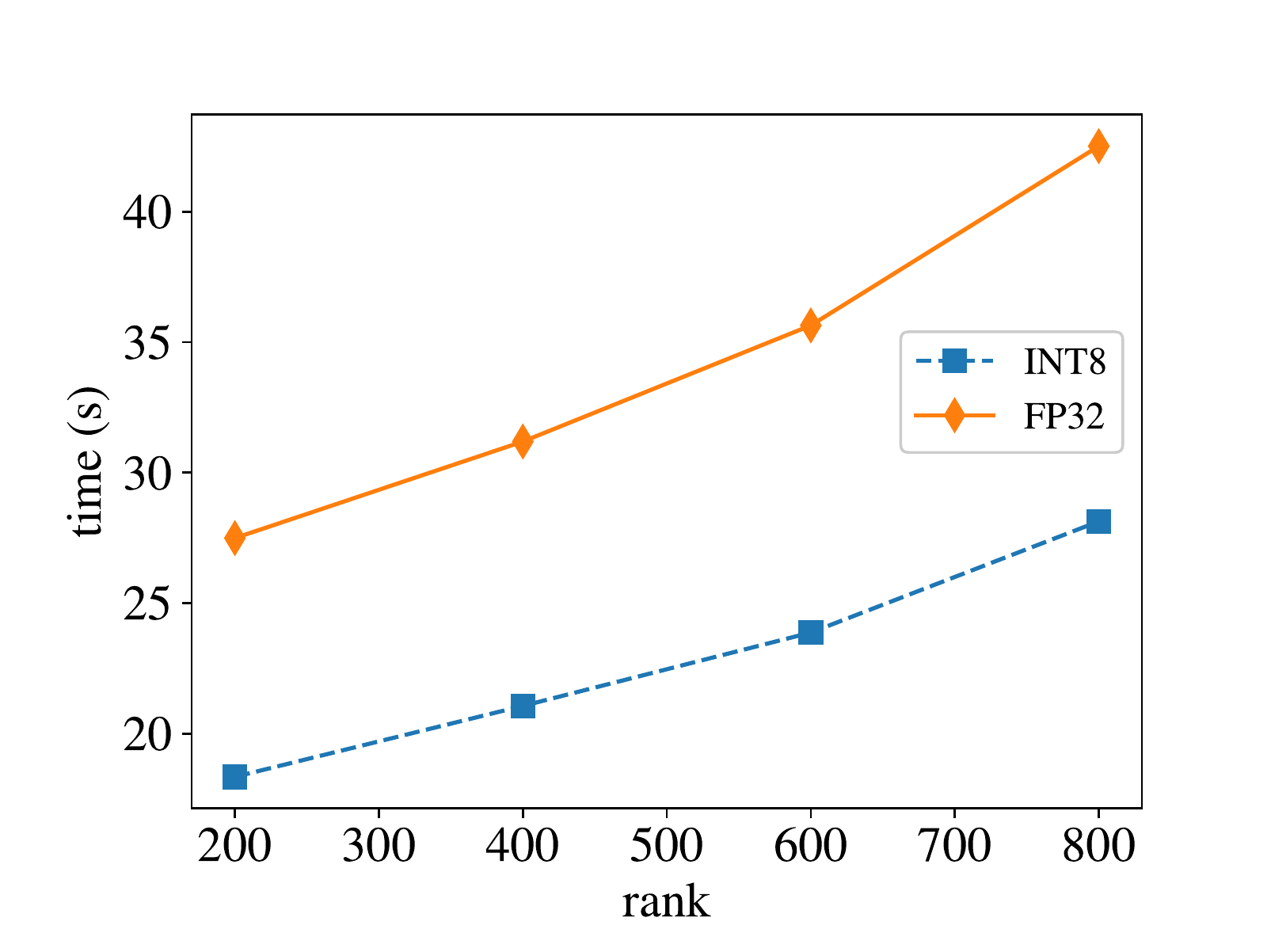}}
    \subfloat[Size (200,200,200,200)]{\label{fig:size 200}\includegraphics[width=0.33\textwidth]{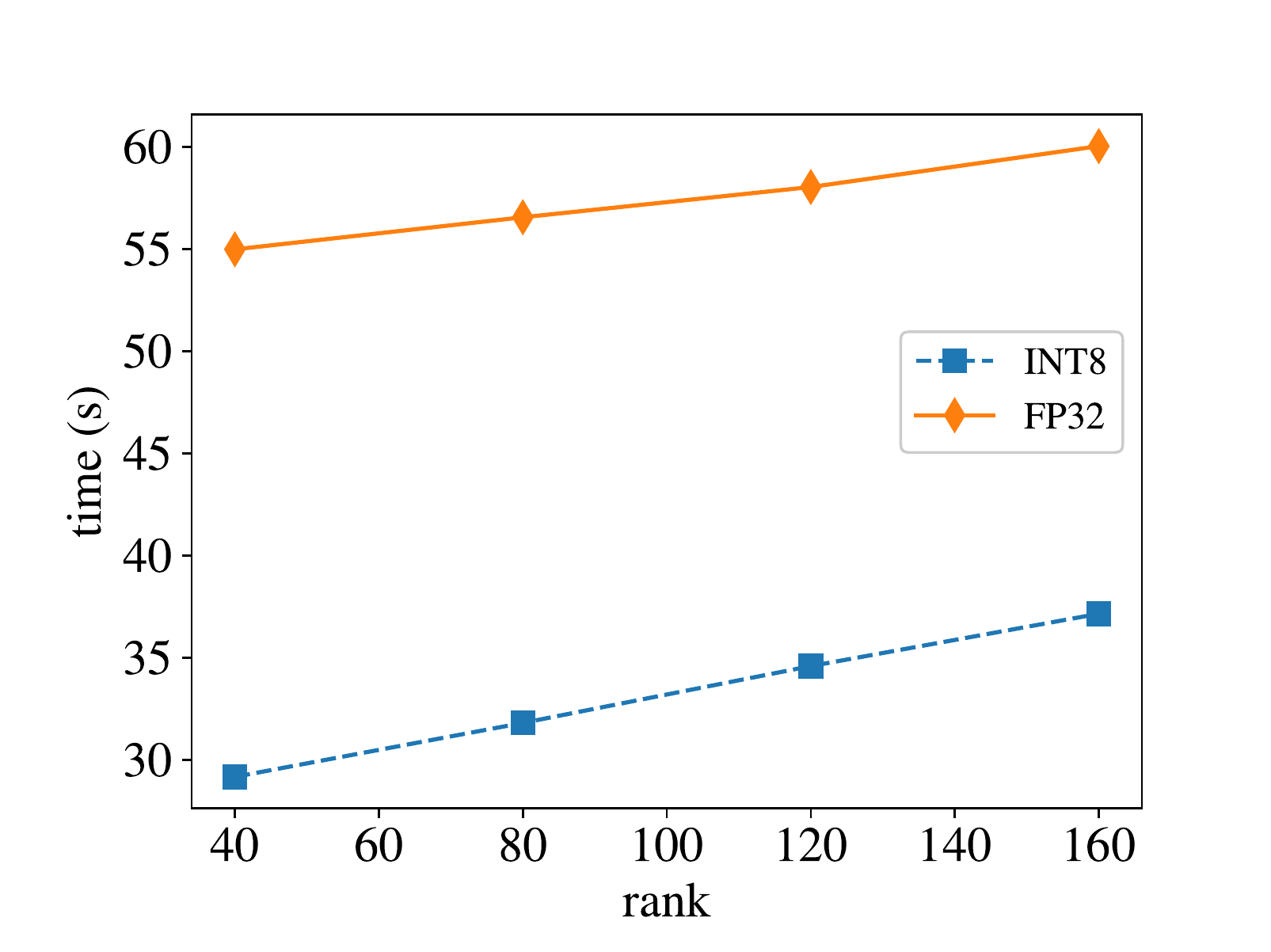}}
    \subfloat[Size (60,60,60,60,60)]{\label{fig:size 60}\includegraphics[width=0.33\textwidth]{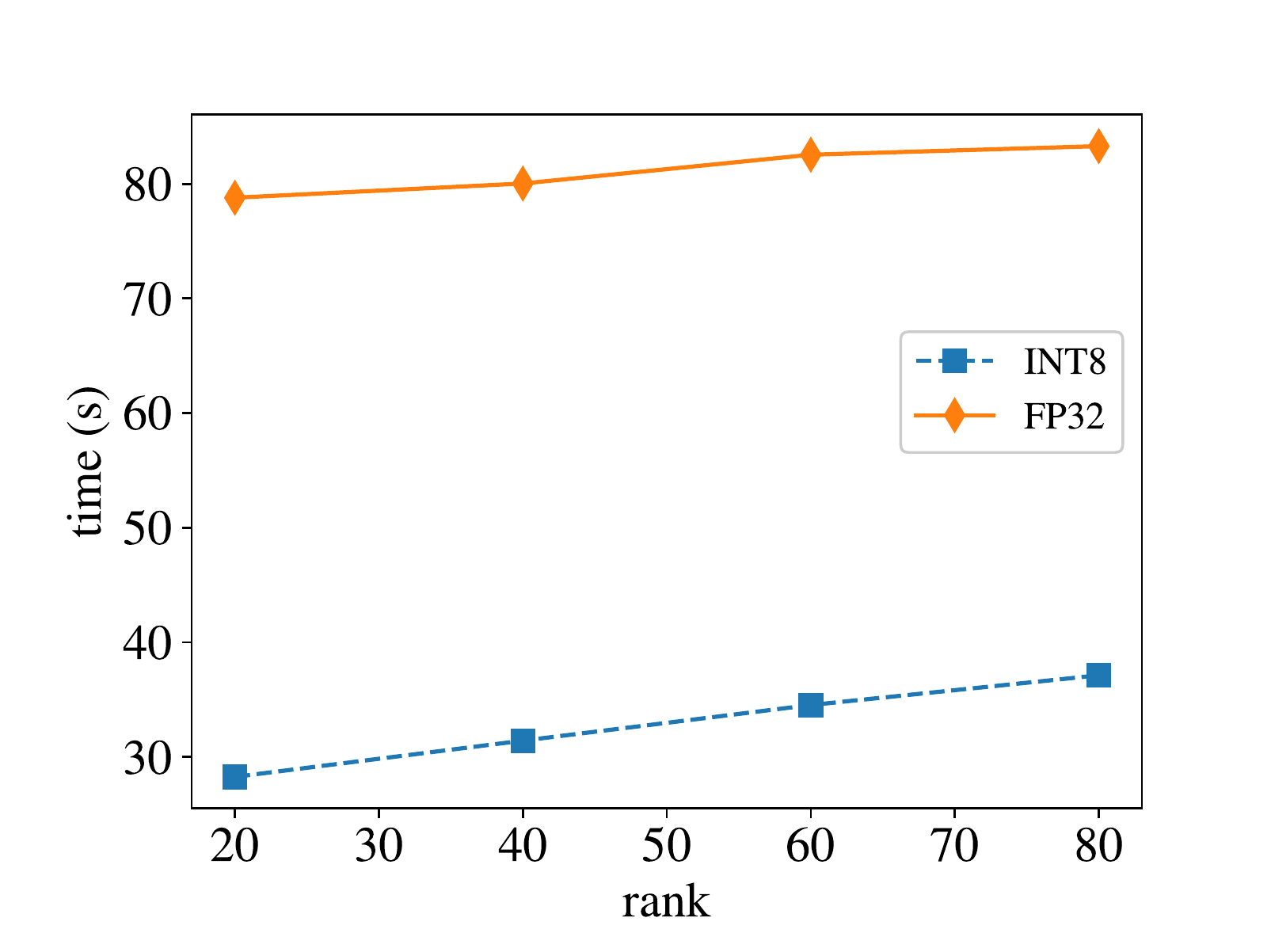}}
\caption{Time comparison between \texttt{FP}$_{32}$ and \texttt{INT}$_8$ of Algorithm \ref{alg:SGD} for various ranks}
\label{figure:time compare}
\end{figure}

\subsubsection{Convergence Comparison in Different Precisions}
We further evaluate the convergence of Algorithm \ref{alg:SGD} under various precisions. We compare precisions \texttt{INT}$_2$, \texttt{INT}$_4$, \texttt{INT}$_8$, and \texttt{FP}$_{32}$, where the computation of \texttt{INT}$_2$ and \texttt{INT}$_4$ is simulated by \texttt{FP}$_{32}$. The simulation simply rounds the scaled number into the nearest integer and then clamps it into the representation range. The convergence of \texttt{INT}$_4$ and \texttt{INT}$_8$ precision are almost the same as the convergence of \texttt{FP}$_{32}$, so they are combined in Figure \ref{figure:precision compare}. The final relative error of low-precision Algorithm \ref{alg:SGD} is slightly worse than the full-precision version due to the quantization error. The quantization error also causes the slow convergence for rank 200 and the divergence for higher ranks of \texttt{INT}$_2$ precision. The noisy ball term in Corollary \ref{cor:signSGD} for SignSGD depends on the rank $r$ and the quantization error $\sigma_Q$. Therefore, a large rank $r$ and large quantization error $\sigma_Q$ may lead to bad convergence due to the large noisy ball. The mixed-precision SGD part starts at around the 2000th iteration. Figure \ref{figure:precision compare} shows that the mixed-precision SGD has a linear convergence rate which matches the theoretical result in Theorem \ref{thm:SGD}. The slower convergence of the mixed-precision SGD part of \texttt{INT}$_2$ precision in Figure \ref{fig:rank200} is caused by the large $\theta$ in Assumption \ref{assumption:descent} due to the quantization error.

\begin{figure}[t]
     \centering
     \subfloat[rank 200]{\label{fig:rank200}\includegraphics[width=0.4\textwidth]{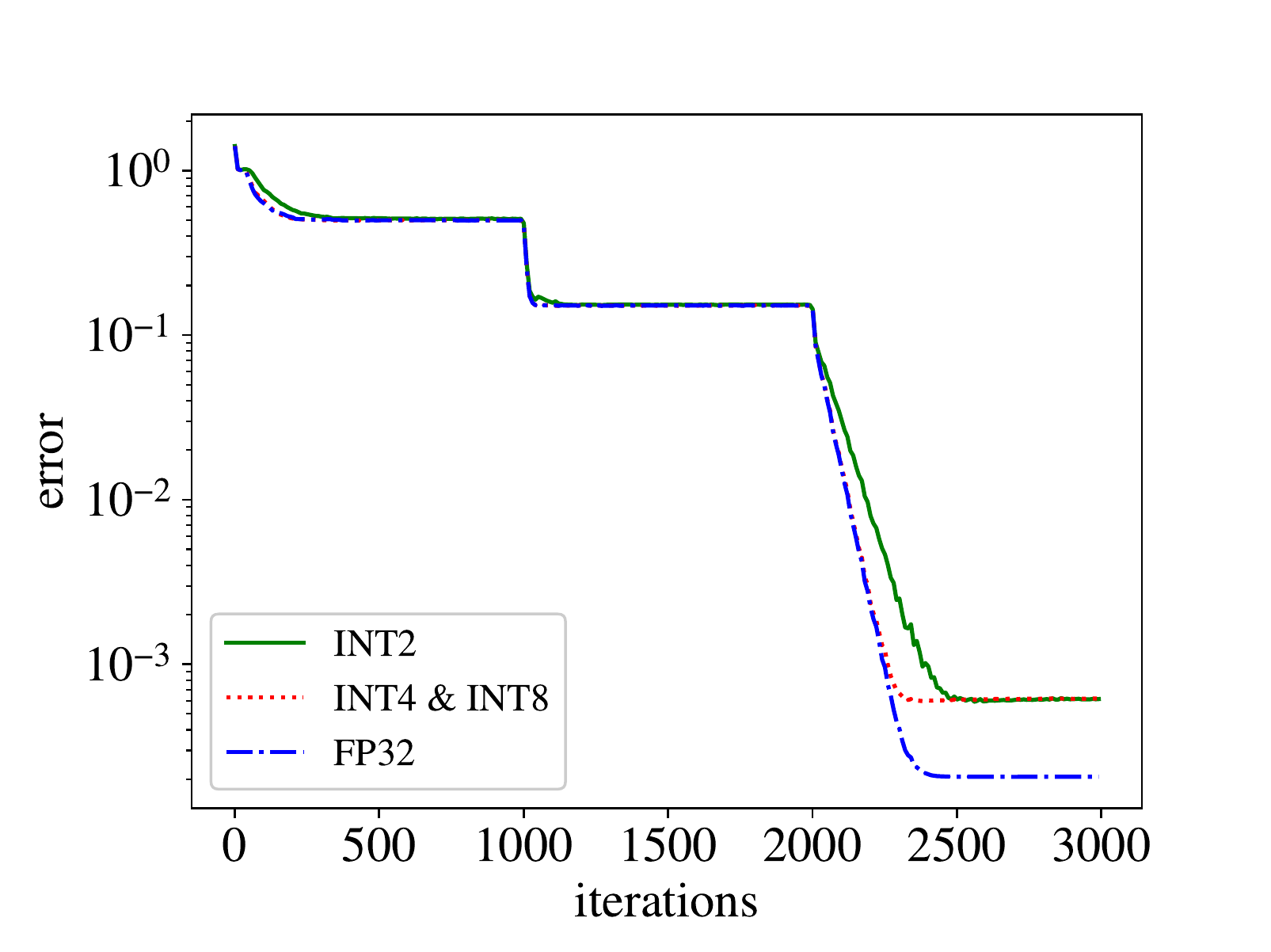}}
    \subfloat[rank 400]{\label{fig:rank400}\includegraphics[width=0.4\textwidth]{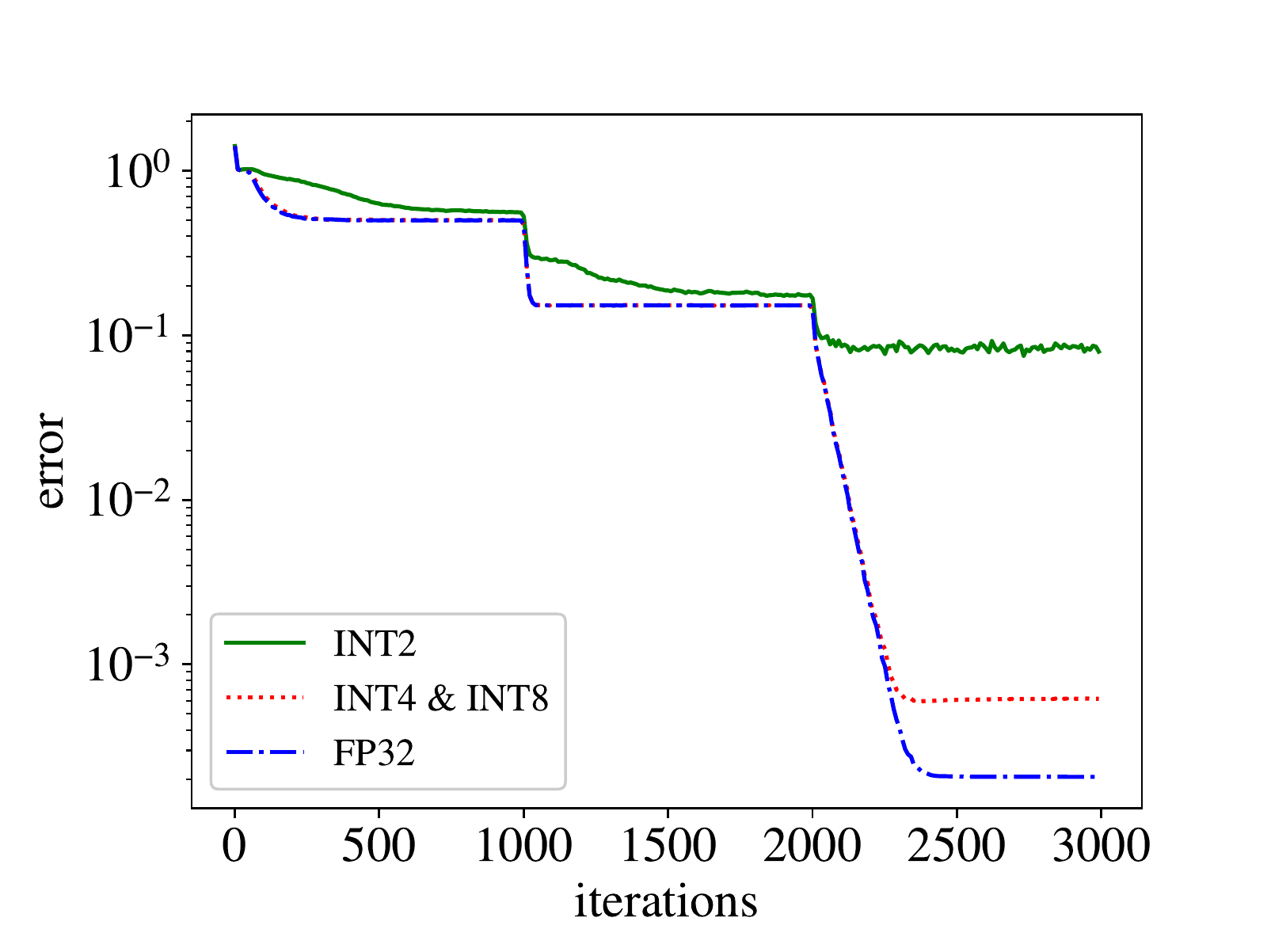}} \\
    \subfloat[rank 600]{\label{fig:rank600}\includegraphics[width=0.4\textwidth]{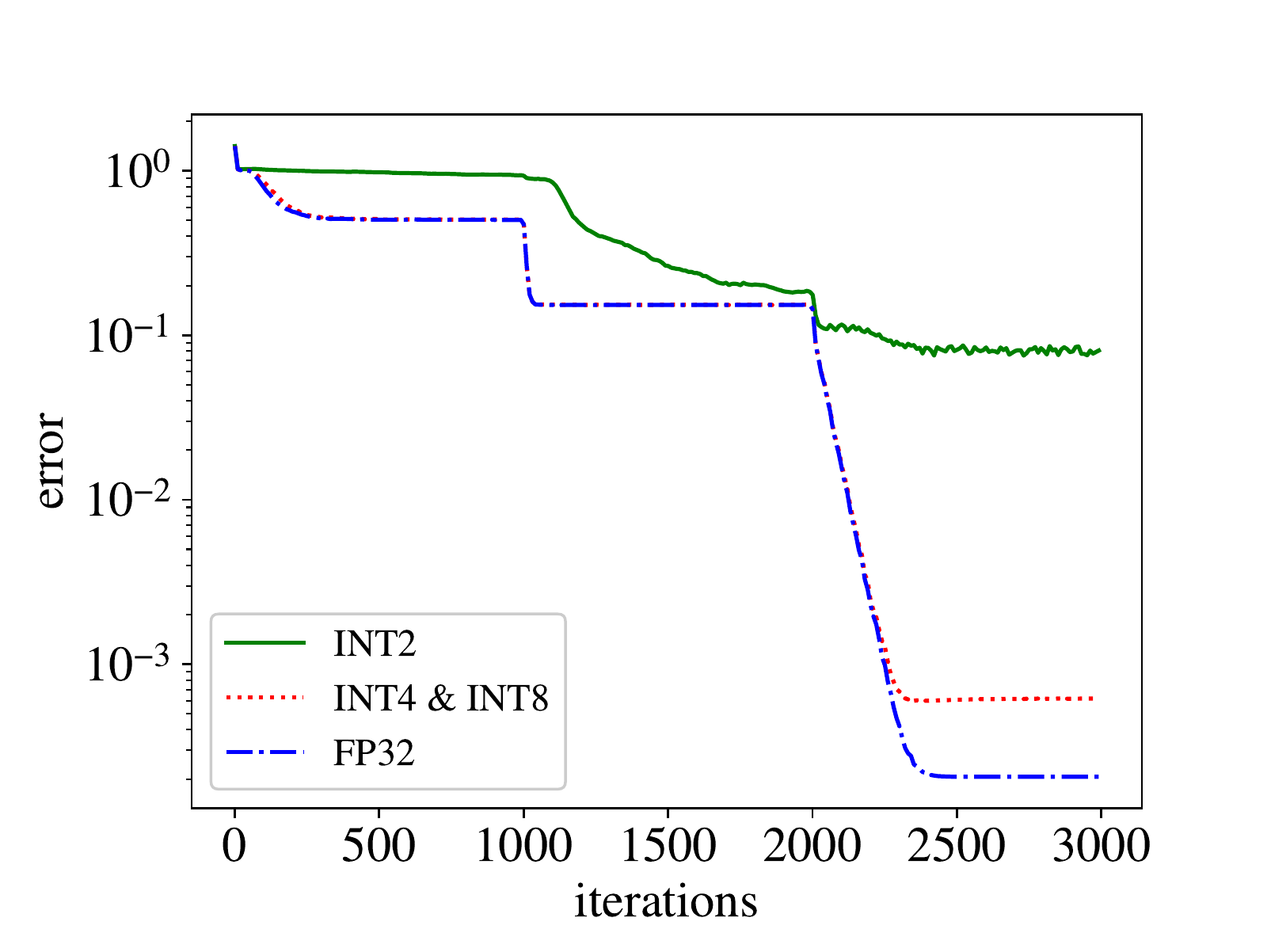}}
    \subfloat[rank 800]{\label{fig:rank800}\includegraphics[width=0.4\textwidth]{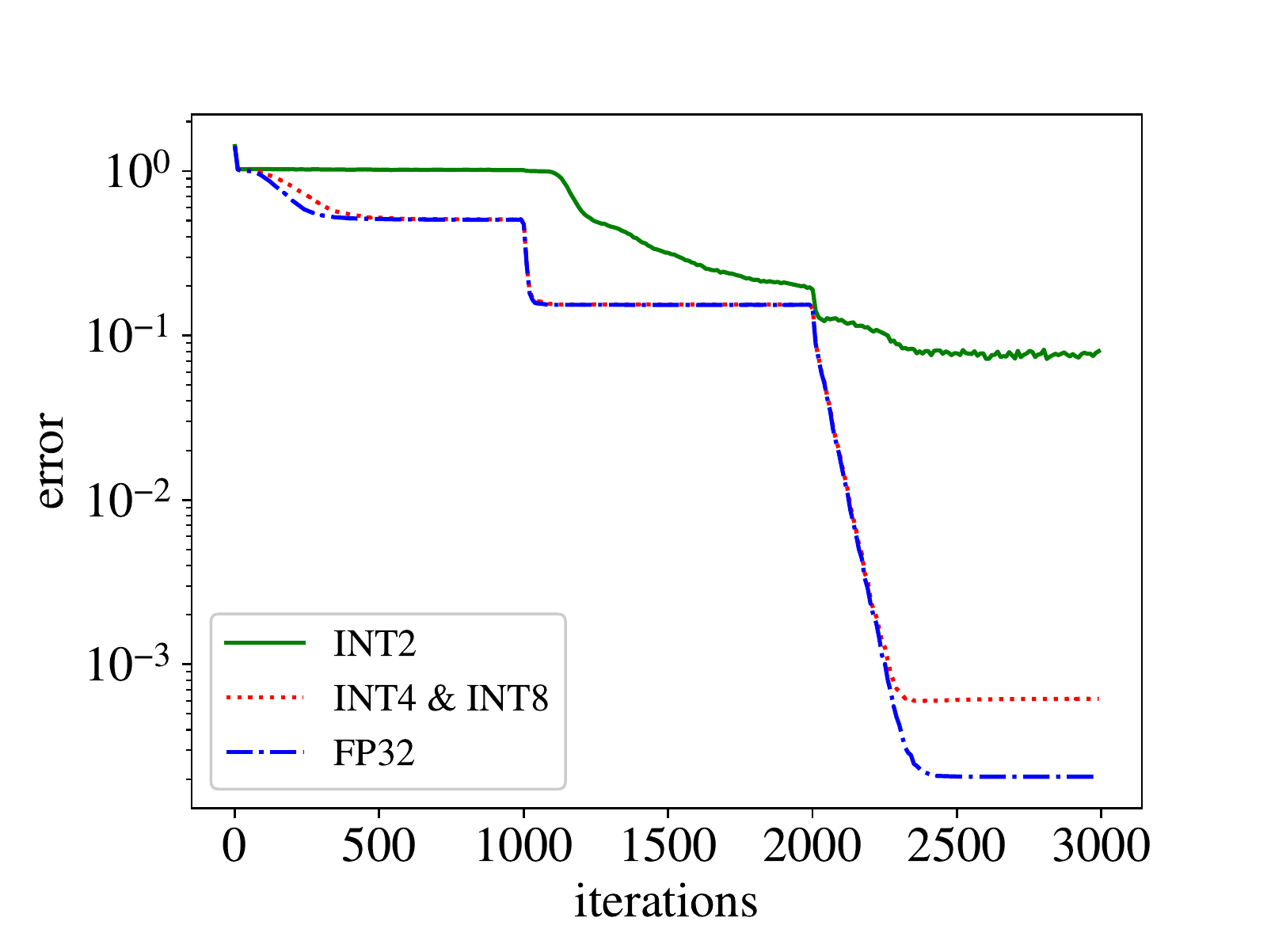}}
\caption{Convergence comparison of various precision on the tensor of dimension (1200,1200,1200).}
\label{figure:precision compare}
\end{figure}

\subsection{Real Datasets}

\subsubsection{Coil-100 Dataset}
The Coil-100 dataset \cite{nene1996columbia} contains the images of 100 objects in 72 different poses. Each image has size $128\times 128\times 3$, where $128\times 128$ is the number of pixels and $3$ represents the $3$ RGB channels. Thus, the size of the formed tensor $\ten{A}$ is $128\times 128\times 3\times 7200$. The CP decomposition is applied for dimension reduction. The fourth-factor matrix $\mat{U}_4 \in \re^{7200\times r}$ can be used as features for clustering and classification tasks.

We run Algorithm \ref{alg:SGD} on the tensor $\ten{A}$ with precisions \texttt{INT}$_2$, \texttt{INT}$_4$, \texttt{INT}$_8$, and \texttt{FP}$_{32}$. The test employs rank $r=16$ and sample size $(32,32,3,1440)$. As shown in Figure \ref{fig:coil_100}, the convergence trends are similar for all precisions. Higher numerical precisions produce smaller relative errors in the final solution, but the difference is insignificant. All of our relative errors are better than the best reported result in \cite{battaglino2018practical}, which is $0.314$. Regarding the running time, \texttt{INT}$_8$ takes 11 seconds for 1000 iterations while \texttt{FP}$_{32}$ takes 31 seconds. Algorithm \ref{alg:SGD} of the \texttt{INT}$_8$ precision is \textbf{2.8} times faster. The experiment demonstrates that our proposed mixed-precision algorithm can effectively reduce the computation cost of CP tensor decomposition on real-world datasets with negligible accuracy loss.

         

\begin{figure}[t]
    \centering
    \includegraphics[width=0.6\textwidth]{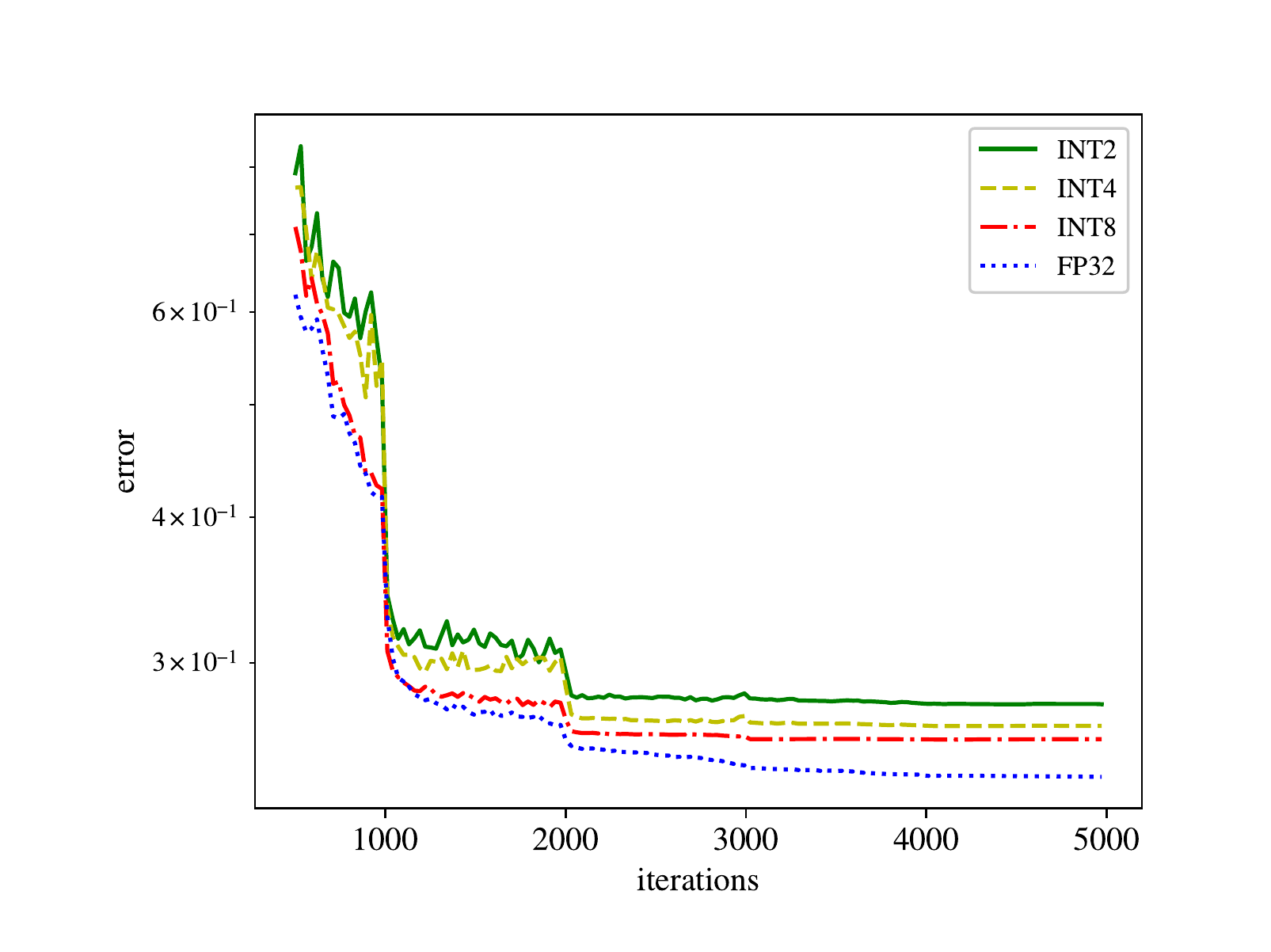}
    \caption{Converge curves of Algorithm \ref{alg:SGD} in various precisions on Coil-100 data}
    \label{fig:coil_100}
\end{figure}

\subsubsection{MRI Dataset}
Magnetic Resonance Imaging (MRI) is widely used in brain science and clinic diagnosis. Low-rank tensor decomposition can be applied to denoise practical MR images \cite{yaman2019low}. This experiment uses the data from the NYU fastMRI Initiative database \cite{knoll2020fastmri,zbontar2018fastmri}. The original data is in a Fourier space and forms a complex tensor $\ten{K}$ of size $16\times 640\times 320$. The tensor $\ten{K}$ in real-world is typically corrupted by noises. In our test, we intentionally corrupt $\ten{K}$ by the noise tensor $\ten{N}$. The real part and imaginary part of $\ten{N}$ both obey the normal distribution with mean $0$ and variance $\tau^2$. Let the corrupted tensor be $\widehat{\ten{K}}=\ten{K}+\ten{N}$, then the inverse Fourier transform is applied to $\widehat{\ten{K}}$ to get $\widehat{\cA}\in \mathbb{C}^{16\times 640\times 320}$. Next, we use Algorithm \ref{alg:SGD} to find a low-rank approximation $\cA$ of the noisy tensor $\widehat{\cA}$ for noise removal. Finally, the gray-scale image $\mat{M}\in \re^{640\times 320}$ is reconstructed as 
\[
    \mat{M}_{i,j} = \sqrt{\sum_{k=1}^{16} |\ten{A}(k,i,j)|^2}.
\]
The image $\mat{M}$ is further cropped into size $320\times 320$ by selecting $\mat{M}(161:480,:)$.
Let $\mat{M}_\text{truth}$ be the ground-truth of the image, then the relative error is computed as 
\[
    \text{error}=\frac{\|\mat{M}-\mat{M}_\text{truth}\|}{\|\mat{M}_\text{truth}\|}.
\]
This experiment applied Algorithm \ref{alg:SGD} to the complex tensors. We would like to remark that Algorithm \ref{alg:SGD} is designed for real tensors, but we can extend the algorithm to complex tensors by considering the real part and the imaginary part separately. 

\begin{table}[t]
        \centering
        \begin{tabular}{|c|c|c|c|c|}
        \hline
          $\tau$  & w/o CP & \texttt{INT}$_4$ & \texttt{INT}$_8$ & \texttt{FP}$_{32}$ \\
          \hline
          1.0e-5 & 0.379 &0.115 & 0.109 & 0.109\\
          \hline
          1.5e-5 & 0.639&0.173 & 0.164 & 0.165 \\
          \hline
          2.0e-5 & 0.913& 0.246 & 0.236 & 0.236 \\ \hline
        \end{tabular}
        \caption{Relative errors before and after removing noises of MRI by Algorithm \ref{alg:SGD}}
        \label{tab:MRI error}
      \end{table}

\captionsetup[subfloat]{labelformat=empty}
\begin{figure}[t]
     \centering
     \rotatebox[origin=l]{90}{\tiny{Noise level $\tau$: $1.0e-5$}}
     \subfloat{\includegraphics[width=0.22\textwidth]{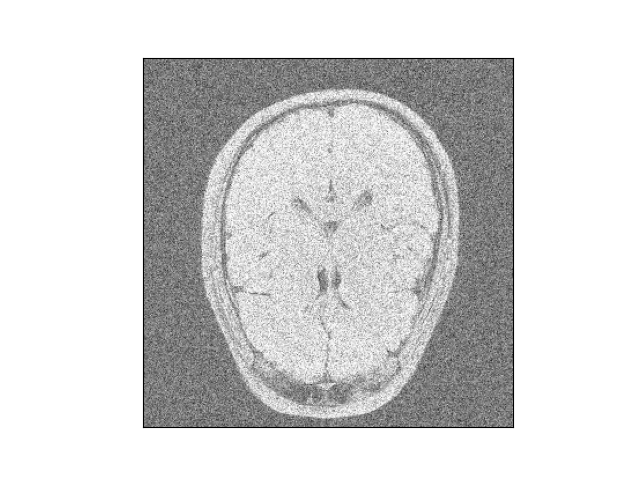}}
    \subfloat{\includegraphics[width=0.22\textwidth]{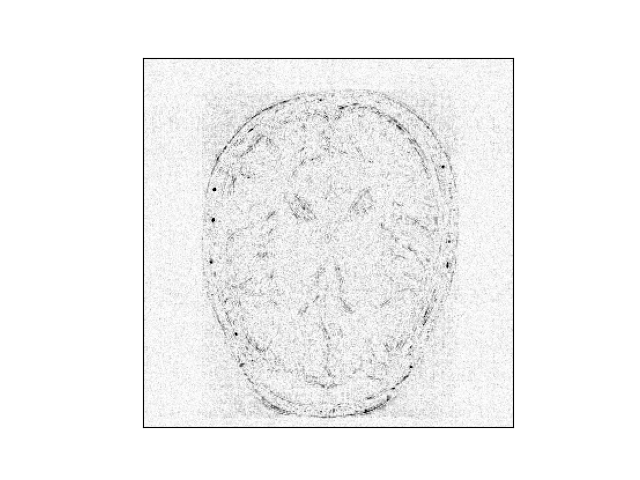}} 
    \subfloat{\includegraphics[width=0.22\textwidth]{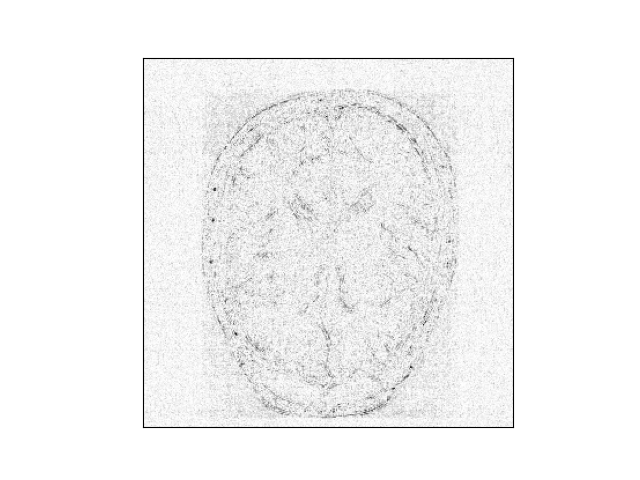}} 
    \subfloat{\includegraphics[width=0.22\textwidth]{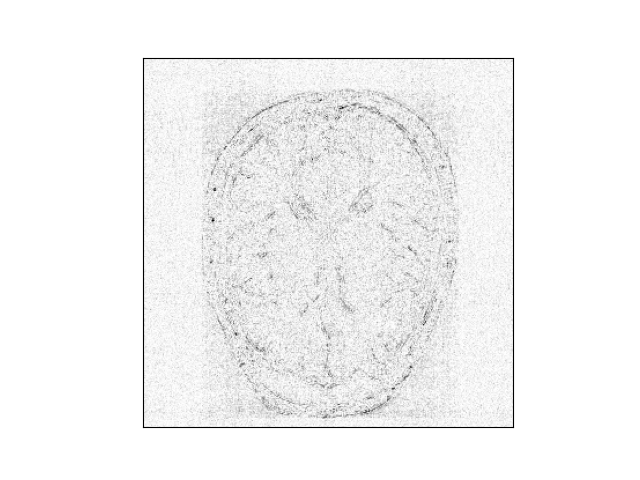}} \\
    \rotatebox[origin=l]{90}{\tiny{Noise level $\tau$: $1.5e-5$}}
     \subfloat{\includegraphics[width=0.22\textwidth]{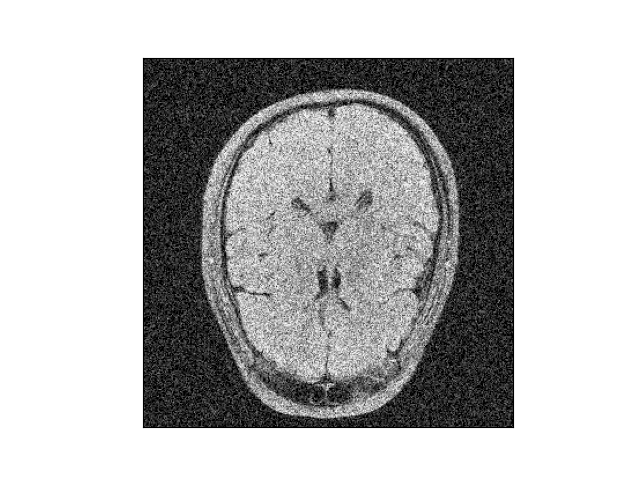}}
    \subfloat{\includegraphics[width=0.22\textwidth]{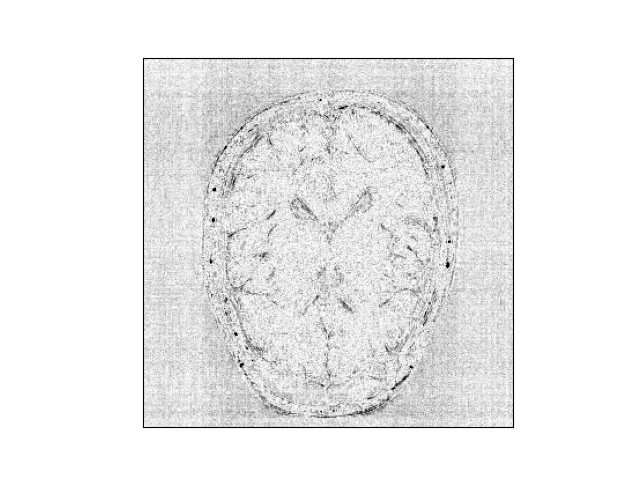}} 
    \subfloat{\includegraphics[width=0.22\textwidth]{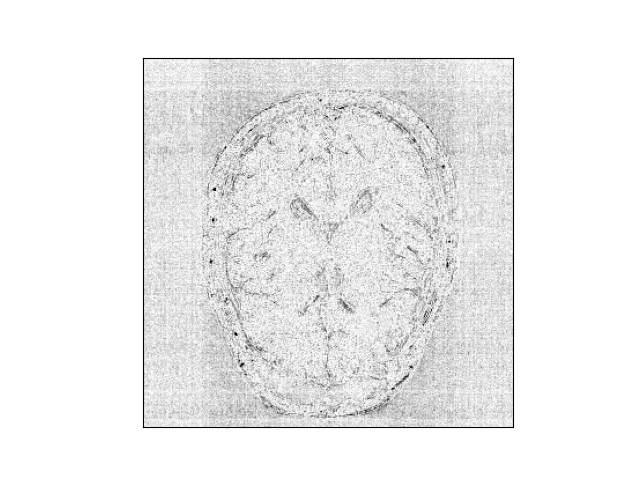}} 
    \subfloat{\includegraphics[width=0.22\textwidth]{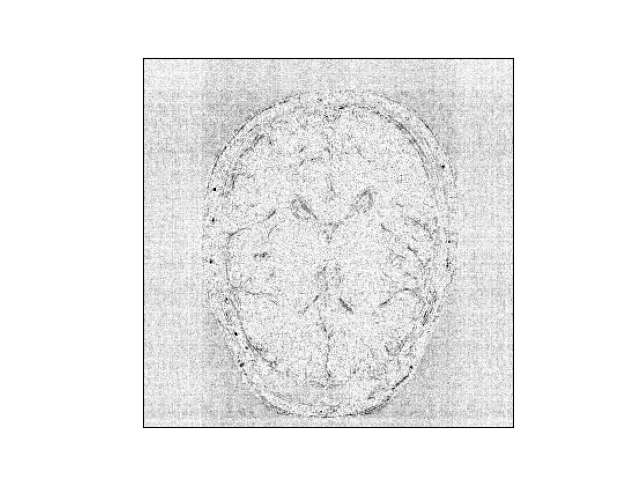}}
    \\
    \rotatebox[origin=l]{90}{\tiny{Noise level $\tau$: $2.0e-5$}}
     \subfloat[w/o CP]{\includegraphics[width=0.22\textwidth]{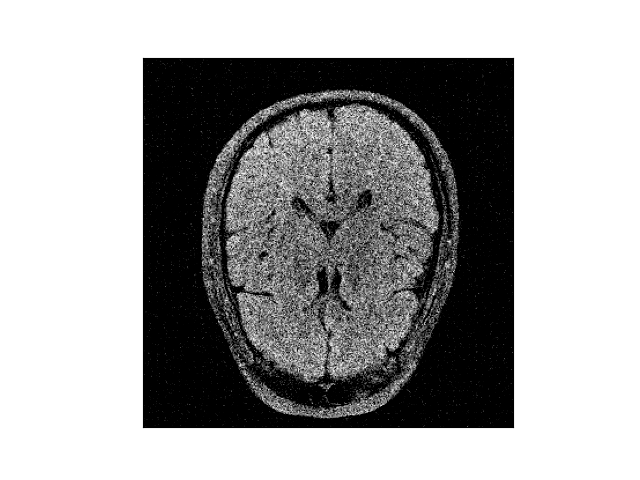}}
    \subfloat[$\texttt{INT}_4$]{\includegraphics[width=0.22\textwidth]{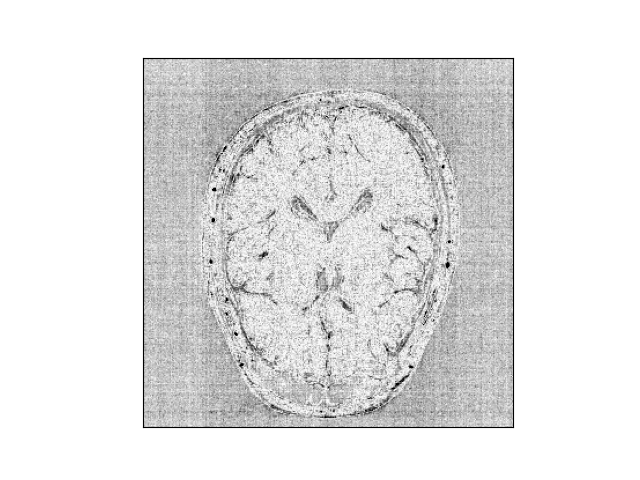}} 
    \subfloat[$\texttt{INT}_8$]{\includegraphics[width=0.22\textwidth]{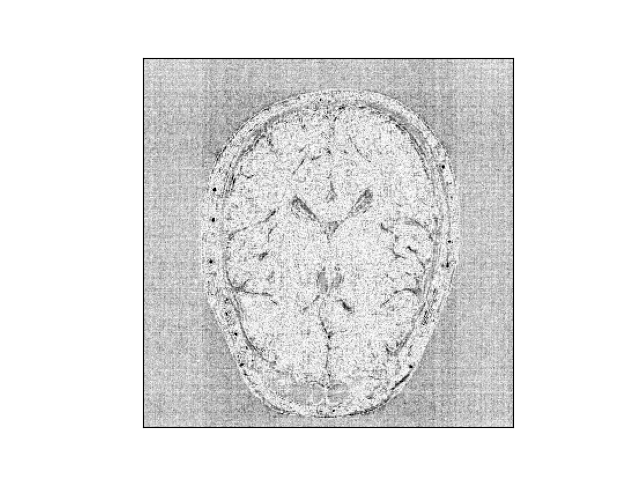}} 
    \subfloat[\texttt{FP}$_{32}$]{\includegraphics[width=0.22\textwidth]{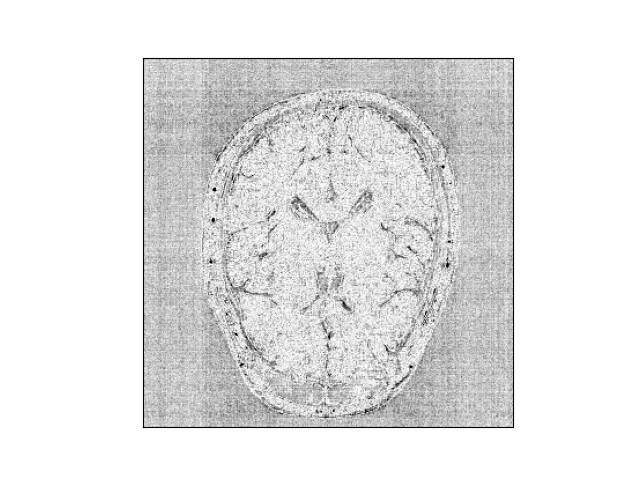}}

\caption{Error images before and after removing noises of MRI by Algorithm \ref{alg:SGD}}
\label{fig:MRI error}
\end{figure}
\captionsetup[subfloat]{labelformat=parens}

Our experiment uses rank $r=200$, sample size $(16,64,64)$ and noise tensors $\ten{N}$ of standard deviations $\tau=1.0\times 10^{-5},1.5\times 10^{-5},2.0 \times 10^{-5}$, respectively. We test the performance of Algorithm \ref{alg:SGD} in various precisions. The test results are presented in Table \ref{tab:MRI error} and Figure \ref{fig:MRI error}. Table \ref{tab:MRI error} lists the relative errors before and after CP decompositions. Algorithm \ref{alg:SGD} successfully removes the noises and reduces the errors. Moreover, the performance of Algorithm \ref{alg:SGD} in \texttt{INT}$_4$ and \texttt{INT}$_8$ is similar to \texttt{FP}$_{32}$. Figure \ref{fig:MRI error} compares the recovered images and the ground-truth image. The images obtained via \texttt{INT}$_4$, \texttt{INT}$_8$, \texttt{FP}$_{32}$ mixed-precision CP decomposition are visually identical and vastly superior to those without noise removal. The results demonstrate that the mixed-precision Algorithm \ref{alg:SGD} is capable of producing accurate decomposition for noisy MRI datasets.

\subsection{FPGA Demonstration for Edge Computing} The proposed mixed-precision CP decomposition can reduce the computing cost on both cloud and edge devices. Here we implement Algorithm \ref{alg:SGD} on a Field Programmable Gate Array (FPGA) to demonstrate its benefit on resource-constrained edge devices. FPGAs are widely used for edge computing due to their energy efficiency, flexible reconfigurability, and fast time-to-market \cite{shan2019exact}. However, FPGAs have very limited memory and computing resources, therefore, it is desired to use low numerical precision to save the hardware cost in massive engineering applications. 

We consider a rank-$20$ tensor with size $(100,100,100)$ and sample size $(20,20,20)$. Table \ref{tab:FPGA} shows the performance of Algorithm \ref{alg:SGD} on FPGA in different precisions. The integer operations accelerate Algorithm \ref{alg:SGD} about \textbf{2.8} times compared to \texttt{FP}$_{32}$. The running time of Algorithm \ref{alg:SGD} in \texttt{INT}$_2$ and \texttt{INT}$_8$ are dominated by higher precision parts in the algorithm. Consequently, \texttt{INT}$_2$ and \texttt{INT}$_8$ have the similar running time. The usage of BRAM (block random-access memory) of \texttt{INT}$_2$ and \texttt{INT}$_8$ is about \textbf{20} times less than \texttt{FP}$_{32}$. Due to the reduced hardware resource requirements, the power consumption of \texttt{INT}$_2$ and \texttt{INT}$_8$ is also reduced. The energy cost of \texttt{FP}$_{32}$ is 4 times more than the energy consumed by \texttt{INT}$_2$ and \texttt{INT}$_8$. The number of FF (Flip-Flop) used by \texttt{INT}$_2$ is \textbf{2.3} times less than \texttt{INT}$_8$ and \textbf{7.2} times less than \texttt{FP}$_{32}$, respectively. In summary, the reduction of time and resources on FPGAs successfully demonstrates the effectiveness of our proposed mixed-precision algorithm on resource-constrained devices.



\begin{table}[t]
    \centering
    \begin{tabular}{|c|c|c|c|c|c|}
    \hline
         Precision & Time (s) & BRAM & FF & Power (W) & Energy (J) \\
         \hline
         \texttt{INT}$_2$ & \textbf{3.08} & \textbf{12} & \textbf{86025} & \textbf{8.4} & \textbf{25.87} \\
         \hline
         \texttt{INT}$_8$ & \textbf{3.08} & 15 & 197868 & 9.5 & 29.26 \\
         \hline
         \texttt{FP}$_{32}$ & 8.61 & 250 & 624416 & 13.3 & 114.4 \\ \hline
    \end{tabular}
    \caption{FPGA implementations of Algorithm \ref{alg:SGD} in different precisions}
    \label{tab:FPGA}
\end{table}

\section{Conclusion}
This paper has proposed a mixed-precision stochastic gradient method for the CP tensor decomposition problem. First, the stochastic gradient is computed in mixed-precision to reduce the runtime and computation cost. Then, we develop a two-stage optimization algorithm to solve the CP decomposition problem using mixed-precision gradients. The convergence of the proposed algorithm has been proved. We have shown that the CP decomposition problem is locally strongly convex after proper normalization. Consequently, the mixed-precision SGD stage in our algorithm can have a linear convergence rate. A set of numerical experiments on GPUs and on edge devices have successfully demonstrated that our mixed-precision algorithm can significantly reduce the computation costs and latency compared to the full-precision algorithm while maintaining high accuracy. The proposed mixed-precision stochastic gradient method can be applied to many gradient-based CP decomposition algorithms. It will be an interesting future topic to study the applications of mixed-precision gradients to other optimization algorithms.

\appendix
\section{Proof of Theorem \ref{thm:convex}}\label{sec:proof of convexity}
In the appendix, we will give the concrete proof of Theorem \ref{thm:convex}. We first show the Hessian is positive definite if and only the Jacobian matrix has full column rank. Then, we prove that the Jacobian matrix has a full column rank in general.

We consider the following vector-valued function
\begin{equation} \label{eq:tF}
    {\tmat{h}}(\tbTheta) = \text{vec}\left([\![\mat{U}_1 , \big [\mat{1}^T;\tmat{U}_2\big] , \cdots , \big[\mat{1}^T;\tmat{U}_m\big] ]\!]\right),
\end{equation}
where $\text{vec}(\cdot)$ is the vectorization function.
It holds that $\tilde{f}(\tbTheta) = \|{\tmat{h}}(\tbTheta)-\text{vec}(\cA)\|^2$.

\begin{lemma} \label{lemma:pdH}
    Suppose that $\tilde{f}(\tbTheta)=0$, then $\nabla^2 \tilde{f}(\tbTheta)$ is positive definite if and only if the Jacobian $\mat{J}_{\tmat{h}}(\tbTheta)$ has full column rank.
\end{lemma}
\begin{proof}
   $\tilde{f}(\tbTheta)=0$ implies that ${\tmat{h}}(\tbTheta)-\text{vec}(\cA)=\mat{0}$. Then, we have
   \[
        \nabla \tilde{f}(\tbTheta) = 2 (\mat{J}_{\tmat{h}}(\tbTheta))^T{\tmat{h}}(\tbTheta), \, \nabla^2\tilde{F}(\tbTheta) = 2 (\mat{J}_{\tmat{h}}(\tbTheta))^T \mat{J}_{\tmat{h}}(\tbTheta).
   \]
    As a result, $\nabla^2\tilde{f}(\tbTheta)$ is positive definite if and only $\mat{J}_{\tmat{h}}(\tbTheta)$ has full column rank.
\end{proof}


In the following, we first prove $\mat{J}_{\tmat{h}}(\tbTheta)$ generally has full column rank for the order $m=3$ and the even $N_1$, then we extend the result to general orders and dimensions. 

\begin{proposition} \label{prop:convex3}
    When $m=3$, $N_1$ is even, $N_1\ge N_2\ge N_3\ge 3$, and $r\le N_1 \lfloor \frac{n_2n_3}{N_1+N_2+N_3-2} \rfloor$, the Jacobian matrix $\mat{J}_ {\tmat{h}}(\tbTheta)$ has full column rank for generic $\tbTheta = (\mat{U}_1,\tmat{U}_2,\tmat{U}_3 )$.
\end{proposition}
\begin{proof}
    It suffices to only consider $r=N_1 \lfloor \frac{n_2n_3}{N_1+N_2+N_3-2} \rfloor$.
    The Jacobian matrix $\mat{J}_{\tmat{h}}(\tbTheta)$ has full column rank if and only 
    $(\mat{J}_{\tmat{h}}(\tbTheta))^T\mat{J}_{\tmat{h}}(\tbTheta)$ has non-zero determinant, which is a polynomial function $p$ in terms of variables $\tbTheta$. The conclusion is then equivalent to that $p(\tbTheta)$ is nonzero for generic $\tbTheta$. Thus, it suffices to show $p(\tbTheta)$ is not the constant $0$ polynomial \cite{cox2013ideals}. In the following, we will construct the specific $\tbTheta$ such that $p(\tbTheta)\neq 0$. 
    
    Here we denote $\mat{U}_i = [\mat{1}^T;\tmat{U}_i]$ for convenience.
    Let $k =  \frac{N_1}{2}  $ and $R = 2 \lfloor \frac{n_2n_3}{N_1+N_2+N_3-2} \rfloor$.
    We evenly split $\{1,\ldots,r\}$ into $k$ groups such that each group has $R=2 \lfloor \frac{n_2n_3}{N_1+N_2+N_3-2}\rfloor $ elements. We denote the groups by $\group_1,\ldots,\group_k$. For $i=1,\ldots,k$, let 
     \begin{eqnarray*}
        \mat{J}_i&=&\bigg{[}
            \Big[ [\mat{e}_{2i-1},\mat{e}_{2i}]\otimes \mat{U}_2(:,j)\otimes \mat{U}_3(:,j)\Big]_{j=1}^r,\, \left[\mat{U}_1(:,j)\otimes \tilde{\mat{I}}_{N_2} \otimes \mat{U}_3(:,j)\right]_{j\in \group_i}, \\
            && \qquad \qquad \qquad  \left[\mat{U}_1(:,j)\otimes \mat{U}_2(:,j) \otimes \tilde{\mat{I}}_{N_3}\right]_{j\in \group_i} \bigg],
    \end{eqnarray*}
    where $\tmat{I}_n:=\mat{I}_n(2:n,:)$. The Jacobian $\mat{J}_{\tmat{h}}(\tbTheta)$ can be written as  $\big[\mat{J}_1,\ldots,\mat{J}_k\big]$.
   
    We construct the matrix $\mat{U}_1$, where for $j\in \group_i$
    \[
        \mat{U}_1(l,j) = 
            0,  \text{ if } l \neq 2i-1, 2i.
    \]
    For such $\mat{U}_1$, the Jocabian $\mat{J}_{\tmat{h}}(\tbTheta)$ has full column rank if and only if each $\mat{J}_i$ has full column rank. In the following, we will construct $\tbTheta$ such that $\mat{J}_1$ has full column rank. 

    Since $N_1\ge N_2 \ge N_3 \ge 3$, it holds that 
    \[
        R = 2 \lfloor \frac{n_2n_3}{N_1+N_2+N_3-2} \rfloor \le \frac{2n_2n_3}{2n_2+N_3-2}\le \frac{2n_2n_3}{2n_2-1}< \frac{2n_2n_3}{2n_2} = N_3.
    \]
    Thus, we have $R<N_3\le N_2$.
    Let $\mat{a}_1,\ldots,\mat{a}_r$ be pairwisely independent vectors, $\mat{b}_1,\ldots,\mat{b}_{N_2}$ and $\mat{c}_1,\ldots,\mat{c}_{N_3}$ be orthonormal basis of $\re^{N_2}$ and $\re^{N_3}$, respectively. The orthonormal basis can be chosen such that all leading entries are nonzero. Denote the matrix
    \[
        \mat{P} := \left[
            \left[\mat{I}_2\otimes \mat{b}_j \otimes \mat{c}_j, \mat{a}_j \otimes \tmat{I}_{N_2} \otimes \mat{c}_j, \mat{a}_j \otimes \mat{b}_j \otimes \tmat{I}_{N_3}\right]_{j=1}^R, \mat{Q}_1,\mat{Q}_2,\mat{Q}_3
        \right],
    \]
    where
    \[
     \mat{Q}_1 = \big[\mat{I}_2\otimes \mat{b}_i \otimes \mat{c}_j \big]
     _{R+1\le i \le N_2,R+1\le j \le N_3},
    \]
    \[
     \mat{Q}_2 = \big[\mat{I}_2\otimes \mat{b}_i \otimes (\mat{c}_{2j-1}+\mat{c}_{2j})\big]_{R+1\le i \le N_2,1\le j \le R/2},
    \]
    \[
     \mat{Q}_3 = \big[\mat{I}_2\otimes (\mat{b}_{2i-1}+\mat{b}_{2i}) \otimes \mat{c}_j\big]_{1\le i \le R/2,R+1\le j \le N_3}.
    \]
    Next, we show the matrix $\mat{P}$ has full column rank. It is equivalent to proving that $\mat{P}\mat{x}=\mat{0} \Leftrightarrow \mat{x}=\mat{0}$. We first prove the coefficients for $\mat{Q}_1$ is zero. For some $R+1\le i \le N_2,R+1\le j \le N_3$, it holds that 
    \begin{eqnarray*}
        \big(\mat{I}_2\otimes \mat{b}_i^T \otimes \mat{c}_j^T\big) \mat{P}\mat{x} = \mat{I}_2 {\gmat{\lambda}} = \mat{0} \Rightarrow \gmat{\lambda} = \mat{0},
    \end{eqnarray*}
    where $\gmat{\lambda}$ is the coefficient vector corresponding to $\mat{I}_2\otimes \mat{b}_i \otimes \mat{c}_j$ in $\mat{x}$. Thus the coefficient for $\mat{Q}_1$ is zero. Then we show the coefficient for $\mat{Q}_2$ is zero. For some $R+1\le i \le N_2,1\le j \le R/2$, it holds 
    \[
        \big(\mat{I}_2\otimes \mat{b}_i^T \otimes \mat{c}_{2j-1}^T \big ) \mat{P}\mat{x} = \big(\mat{a}_{2j-1}\otimes \mat{b}_i^T \tmat{I}_{N_2}, \mat{I}_2\big)  \begin{bmatrix} \gmat{\lambda_1} \\ \gmat{\mu} \end{bmatrix} = \mat{0},
    \]
    \[
        \big(\mat{I}_2\otimes \mat{b}_i^T \otimes \mat{c}_{2j}^T\big) \mat{P}\mat{x} = \big(\mat{a}_{2j}\otimes \mat{b}_i^T \tmat{I}_{N_2}, \mat{I}_2\big)  \begin{bmatrix} \gmat{\lambda}_2 \\ \gmat{\mu} \end{bmatrix} = \mat{0},
    \]
    where $\gmat{\lambda}_1,\gmat{\lambda}_2,\gmat{\mu}$ are coefficients corresponding to $\mat{a}_{2j-1}\otimes \tmat{I}_{N_2} \otimes \mat{c}_{2j},\mat{a}_{2j}\otimes \tmat{I}_{N_2} \otimes \mat{c}_{2j}$ and $\mat{I}_2\otimes \mat{b}_i \otimes (\mat{c}_{2j-1}+\mat{c}_{2j}) $ respectively. By above equations, we know 
    \[
        \gmat{\mu} = -(\mat{a}_{2j-1}\otimes \mat{b}_i^T \tmat{I}_{N_2}) \gmat{\lambda}_1 = -(\mat{a}_{2j}\otimes \mat{b}_i^T \tmat{I}_{N_2}) \gmat{\lambda}_2.
    \]
    Thus, $ \mat{b}_i^T \tmat{I}_{N_2} \gmat{\lambda}_1 = \mat{b}_i^T \tmat{I}_{N_2} \gmat{\lambda}_2 = \mat{0}$ since $\mat{a}_{2j-1},\mat{a}_{2j}$ are linearly independent. It implies that $\gmat{\mu}=0$. Therefore, the coefficient for $\mat{Q}_2$ is zero. Similarly, we can prove the coefficient for $\mat{Q}_3$ is zero by using exactly the same technique. 
    
    Next, we show the coefficient of $\mat{a}_j\otimes \tmat{I}_{N_2}\otimes \mat{c}_j$ is zero. For some $1\le i\le N_2$ and $i\neq j$, it holds 
    \[
        \big(\mat{I}_2\otimes \mat{b}_i^T \otimes \mat{c}_j^T\big)\mat{P}\mat{x} = \big(\mat{a}_j \otimes \mat{b}_i^T \tmat{I}_{N_2},\mat{a}_i \otimes \mat{c}_j^T \tmat{I}_{N_3}\big) \begin{bmatrix}
            \gmat{\mu} \\ \gmat{\lambda}_j
        \end{bmatrix} = \mat{0} \Rightarrow \mat{b}_i^T \tmat{I}_{N_2} \gmat{\mu} = \mat{c}_j^T \tmat{I}_{N_3} \gmat{\lambda}_j = 0,
    \]
    where $\gmat{\mu},\gmat{\lambda}_j$ are coefficients corresponding to $\mat{a}_j \otimes \tmat{I}_{N_2} \otimes \mat{c}_j, \mat{a}_i \otimes \mat{b}_i \otimes \tmat{I}_{N_3}$ respectively. The above equation holds for every $i$ such that $1\le i\le N_2$ and $i\neq j$. Therefore, we have
    \[
        \big[\tmat{I}_{N_2}^T\mat{b}_1,\ldots, \tmat{I}_{N_2}^T\mat{b}_{j-1},\tmat{I}_{N_2}^T\mat{b}_{j+1},\ldots,\tmat{I}_{N_2}^T\mat{b}_{N_2}\big]^T \gmat{\mu} = \mat{0} \Rightarrow \gmat{\mu} = \mat{0}.
    \]
    The above holds because $\big[\tmat{I}_{N_2}^T\mat{b}_1,\ldots, \tmat{I}_{N_2}^T\mat{b}_{j-1},\tmat{I}_{N_2}^T\mat{b}_{j+1},\ldots,\tmat{I}_{N_2}^T\mat{b}_{N_2}\big]^T \in \re^{(N_2-1)\times (N_2-1)} $ is nonsingular. It proves that the coefficient for $\mat{a}_j \otimes \tmat{I}_{N_2} \otimes \mat{c}_j$ is zero. Similarly, we can prove the coefficient for $\mat{a}_j \otimes \mat{b}_j \otimes \tmat{I}_{N_3} $ is zero. 
    
    The only remaining part in $\mat{P}$ is $\big[\mat{I}_2\otimes \mat{b}_j \otimes \mat{c}_j\big]_{j=1}^R $. $\big[\mat{I}_2\otimes \mat{b}_j \otimes \mat{c}_j\big]_{j=1}^R $ has full column rank since $\mat{b}_1,\ldots,\mat{b}_R$ are linearly independent. Thus, the coefficients corresponding to $\big[\mat{I}_2\otimes \mat{b}_j \otimes \mat{c}_j\big]_{j=1}^R $ are also zero. It finishes the proof that $\mat{P}\mat{x}=\mat{0} \Leftrightarrow \mat{x} = \mat{0}$. Thus, $\mat{P}$ has full column rank. 
    
    Let
    \begin{eqnarray*}
        \mat{Q} := 
        \bigg[
        &\Big[\frac{\mat{b}_i}{(\mat{b}_i)_1} \otimes \frac{\mat{c}_j}{(\mat{c}_j)_1}\Big]_{R+1\le i \le N_2,R+1\le j \le N_3}, 
        \Big[\frac{\mat{b}_i}{(\mat{b}_i)_1} \otimes \frac{\mat{c}_{2j-1}+\mat{c}_{2j}}{(\mat{c}_{2j-1})_1+(\mat{c}_{2j})_1}\Big]_{R+1\le i \le N_2,1\le j \le R/2},\\
        &\Big[\frac{\mat{b}_{2i-1}+\mat{b}_{2i}}{(\mat{b}_{2i-1})_1+(\mat{b}_{2i})_1} \otimes \frac{\mat{c}_j}{(\mat{c}_j)_1}\Big]_{1\le i \le R/2,R+1\le j \le N_3}
        \bigg].
    \end{eqnarray*}
    The number of columns of $\mat{Q}$ is 
    \[
        c = (N_2-R)(N_3-R) + (N_2-R)\frac{R}{2} + (N_3-R)\frac{R}{2} = n_2n_3-\frac{R}{2}(N_2+N_3).
    \]
    It holds that,
    \[
        c - (r-R) = n_2n_3-\frac{R}{2}(N_2+N_3-2) - \frac{n_1R}{2} =  n_2n_3 - \frac{R}{2}(N_1+N_2+N_3-2) \ge 0.
    \]
    Let $\mat{U}_1(1:2,j)=\mat{a}_j,\tmat{U}_2(:,j) = \mat{b}_j(2:)/(\mat{b}_j)_1,\tmat{U}_3(:,j) = \mat{c}_j(2:)/(\mat{c}_j)_1$ for $j=1,\ldots,R$ and $\tmat{U}_2(:,j),\tmat{U}_3(:,j)$ be vectors such that 
    \begin{eqnarray*}
        \Big[[1;\tmat{U}_2(:,j)]\otimes [1;\tmat{U}_3(:,j)]\Big]_{R+1\le j \le r} = \mat{Q}(:,1:r-R).
    \end{eqnarray*}
    Then, all columns of $\mat{J}_1$ are from $\mat{P}$. We have shown that $\mat{P}$ has full column rank, so $\mat{J}_1$ must have full column rank. The same technique can be applied to $\mat{J}_1,\ldots,\mat{J}_k$. 
    
   Now we have proven that $\mat{J}_{\tmat{h}}(\tbTheta)$ is has full column rank for some $\tbTheta$. Therefore, the Jacobian $\mat{J}_{\tmat{h}}(\tbTheta)$ has full column rank for generic $\tbTheta=(\mat{U}_1,\tmat{U}_2,\tmat{U}_3)$.
\end{proof}

Proposition \ref{prop:convex3} proves the case when $N_1$ is even. If $N_1$ is odd, we may simply consider $N_1-1$ to make the largest dimension even. The result is stated in the following corollary.
\begin{corollary}
 If $m=3$, $N_1 \ge N_2\ge N_3\ge 3$, and $r\le r_3$ for $r_3$ in \eqref{eq:r3}, then the Jacobian matrix $\mat{J}_{\tmat{h}}(\tbTheta)$ has full column rank for generic $\tbTheta = (\mat{U}_1,\tmat{U}_2,\tmat{U}_3 )$.
\end{corollary}
\begin{proof}
    When $N_1$ is even, it is the result of Proposition \ref{prop:convex3}.
    
    When $N_1$ is odd, we consider the dimension $(\tilde{N}_1,\tilde{N}_2,\tilde{N}_3)$, where $\tilde{N}_{1},\tilde{N}_{2},\tilde{N}_{3}$ are largest integers such that $\tilde{N}_{1}$ is even, $\tilde{N}_{1}\ge \tilde{N}_{2} \ge \tilde{N}_{3}$, and $N_i \ge \tilde{N}_i$ for $i=1,2,3$. Let $\mat{V}_1:=\mat{U}_1(1:\tilde{N}_1 ,:),\tmat{V}_2:=\tmat{U}_2(1:\tilde{N}_2-1 ,:),\tmat{V}_3:=\tmat{U}_3(1:\tilde{N}_3-1 ,:)$ By Proposition \ref{prop:convex3}, the matrix $\mat{J}_{\tmat{h}}(\mat{V}_1,\tmat{V}_2,\tmat{V}_3)$ has full column rank for generic $\mat{V}_1,\tmat{V}_2,\tmat{V}_3$. The matrix $\mat{J}_{\tmat{h}}(\mat{V}_1,\tmat{V}_2,\tmat{V}_3)$ only consists of some rows in $\mat{J}_{\tmat{h}}(\tbTheta)$. Thus, $\mat{J}_{\tmat{h}}(\tbTheta)$ has full column rank if $\mat{J}_{\tmat{h}}(\mat{V}_1,\tmat{V}_2,\tmat{V}_3)$ has full column rank. It proves $\mat{J}_{\tmat{h}}(\tbTheta)$ has full column rank for generic $\tbTheta = (\mat{U}_1,\tmat{U}_2,\tmat{U}_3 )$.
\end{proof}

We have shown the local convexity for order-3 tensors. Finally, we prove the higher order cases by induction on the order $m$ with base case $m=3$.
\begin{theorem} \label{thm:lidJ}
  If $r\le r_m$ for $r_m$ in \eqref{eq:rm}, then the Jacobian matrix $\mat{J}_{\tmat{h}}(\tbTheta)$ has full column rank for generic $\tbTheta=(\mat{U}_1,\tmat{U}_2,\ldots,\tmat{U}_m)$.
\end{theorem}
\begin{proof}
    We will prove the result by induction on the order $m$. Proposition \ref{prop:convex3} proves the conclusion for $m=3$. Suppose that the result holds for the order $m-1$, then we show it also holds for $m$. Similar to the proof of Theorem \ref{prop:convex3}, it suffices to prove there exists some $\tbTheta$ such that the Jacobian is has full column rank \cite{cox2013ideals}.
    
    Denote $L:= \min \{r_{m-1},\lfloor \frac{n_2n_3\cdots N_m}{N_1+\cdots+N_m-m+1}\rfloor \}$ and $\mat{Q}_1=\odot_{i=2}^{m} [\mat{1}^T;\tmat{U}_i(:,1:L)]$,
    \[
        \mat{Q}_k = \left[\otimes_{i=2}^{k-1} [\mat{1}^T;\tmat{U}_i(:,j)] \otimes \tmat{I}_{n_k}     \otimes_{i=k+1}^{m} [\mat{1}^T;\tmat{U}_i(:,j)]\right]_{j=1}^L,\, k=2,\ldots,m.
    \]
    Let $\mat{U}_1(:,i) = \mat{e}_j $ for $(j-1)L+1\le i\le jL $, then (after omitting zero columns) it holds that
    \[
        \mat{J}_{\tmat{h}}(\tbTheta)(1:\Pi_{i=2}^m N_i,:)= \left [\odot_{i=2}^{m} [\mat{1}^T;\tmat{U}_i(:,L+1:r)],\, \mat{Q}_1,\ldots,\mat{Q}_m\right ].
    \]
    Let $\mat{S}:=\left[\mat{Q}_1,\ldots,\mat{Q}_m\right]$. The matrix $\mat{S}$ has the same column space as $\mat{J}_{\tmat{h}}([\mat{1}^T;\tmat{U}_2(:,1:L)],\ldots,\tmat{U}_m(:,1:L))$.  By the induction assumption, we know $\mat{S}$ generically has full column rank since $L\le r_{m-1}$. The number of columns of $\mat{S}$ is $c=(N_2+\cdots+N_m-m+2)L$, then 
    \begin{eqnarray*}
        \Pi_{i=2}^m N_i - c &\ge& \Pi_{i=2}^m N_i - (N_2+\cdots+N_m-m+2)\lfloor \frac{\Pi_{i=2}^m N_i}{N_1+\cdots+N_m-m+1}\rfloor \\
        &\ge &\Pi_{i=2}^m N_i -\Pi_{i=2}^m N_i \frac{N_2+\cdots+N_m-m+2}{N_1+\cdots+N_m-m+1} \\
        &=&\Pi_{i=2}^m N_i(1-\frac{N_2+\cdots+N_m-m+2}{N_1+\cdots+N_m-m+1}) \\
        &=&\Pi_{i=2}^m N_i \frac{N_1-1}{N_1+\cdots+N_m-m+1} \\
        &\ge& (N_1-1) \lfloor \frac{\Pi_{i=2}^m N_i}{N_1+\cdots+N_m-m+1}\rfloor \\
        &=& r-L.
    \end{eqnarray*}
    Thus, we could choose columns $\{\tmat{U}_i(:,L+1:r)\}_{i=1}^m $ such that $\mat{J}_{\tmat{h}}(\tbTheta)(1:\Pi_{i=2}^m N_i,:)$ is linearly independent. It proves that $\mat{J}_{\tmat{h}}(\tbTheta)(1:\Pi_{i=2}^m N_i,:)$ has full column rank generically. The same proof can be applied to $\mat{J}_{\tmat{h}}(\tbTheta)((j-1)\Pi_{i=2}^m N_i+1:j\Pi_{i=2}^m N_i,:),j=1,\ldots,N_1$. Therefore, $\mat{J}_{\tmat{h}}(\tbTheta)$ has full column rank for generic $\tbTheta=(\mat{U}_1,\tmat{U}_2,\ldots,\tmat{U}_m)$.
    
\end{proof}

Now, we are ready to prove Theorem \ref{thm:convex}.

\begin{proof}
    Under the assumption of Theorem \ref{thm:convex}, it holds that $\tilde{f}(\tbTheta^*)=0$. Therefore, Lemma \ref{lemma:pdH} and Theorem \ref{thm:lidJ} imply that the Hessian $\nabla^2 \tilde{f}(\tbTheta^*)$ is positive definite for generic $\tbTheta^*$.
\end{proof}

\bibliographystyle{siamplain}
\bibliography{references.bib}

\begin{thebibliography}{10}

\bibitem{abdelfattah2021survey}
{\sc A.~Abdelfattah, H.~Anzt, E.~G. Boman, E.~Carson, T.~Cojean, J.~Dongarra,
  A.~Fox, M.~Gates, N.~J. Higham, X.~S. Li, et~al.}, {\em A survey of numerical
  linear algebra methods utilizing mixed-precision arithmetic}, The
  International Journal of High Performance Computing Applications, 35 (2021),
  pp.~344--369.

\bibitem{anandkumar2014tensor}
{\sc A.~Anandkumar, R.~Ge, D.~Hsu, S.~M. Kakade, and M.~Telgarsky}, {\em Tensor
  decompositions for learning latent variable models}, Journal of machine
  learning research, 15 (2014), pp.~2773--2832.

\bibitem{battaglino2018practical}
{\sc C.~Battaglino, G.~Ballard, and T.~G. Kolda}, {\em A practical randomized
  {CP} tensor decomposition}, SIAM Journal on Matrix Analysis and Applications,
  39 (2018), pp.~876--901.

\bibitem{bernstein2018signsgd}
{\sc J.~Bernstein, Y.-X. Wang, K.~Azizzadenesheli, and A.~Anandkumar}, {\em
  {signSGD}: Compressed optimisation for non-convex problems}, in International
  Conference on Machine Learning, PMLR, 2018, pp.~560--569.

\bibitem{beutel2014flexifact}
{\sc A.~Beutel, P.~P. Talukdar, A.~Kumar, C.~Faloutsos, E.~E. Papalexakis, and
  E.~P. Xing}, {\em Flexifact: Scalable flexible factorization of coupled
  tensors on hadoop}, in Proceedings of the 2014 SIAM international conference
  on data mining, SIAM, 2014, pp.~109--117.

\bibitem{bigoni2016spectral}
{\sc D.~Bigoni, A.~P. Engsig-Karup, and Y.~M. Marzouk}, {\em Spectral
  tensor-train decomposition}, SIAM Journal on Scientific Computing, 38 (2016),
  pp.~A2405--A2439.

\bibitem{bottou2018optimization}
{\sc L.~Bottou, F.~E. Curtis, and J.~Nocedal}, {\em Optimization methods for
  large-scale machine learning}, Siam Review, 60 (2018), pp.~223--311.

\bibitem{bro1997parafac}
{\sc R.~Bro}, {\em {PARAFAC}. tutorial and applications}, Chemometrics and
  intelligent laboratory systems, 38 (1997), pp.~149--171.

\bibitem{buttari2007mixed}
{\sc A.~Buttari, J.~Dongarra, J.~Langou, J.~Langou, P.~Luszczek, and
  J.~Kurzak}, {\em Mixed precision iterative refinement techniques for the
  solution of dense linear systems}, The International Journal of High
  Performance Computing Applications, 21 (2007), pp.~457--466.

\bibitem{carson2018accelerating}
{\sc E.~Carson and N.~J. Higham}, {\em Accelerating the solution of linear
  systems by iterative refinement in three precisions}, SIAM Journal on
  Scientific Computing, 40 (2018), pp.~A817--A847.

\bibitem{carson2020three}
{\sc E.~Carson, N.~J. Higham, and S.~Pranesh}, {\em Three-precision
  {GMRES}-based iterative refinement for least squares problems}, SIAM Journal
  on Scientific Computing, 42 (2020), pp.~A4063--A4083.

\bibitem{carson2022mixed}
{\sc E.~Carson and N.~Khan}, {\em Mixed precision iterative refinement with
  sparse approximate inverse preconditioning}, arXiv preprint arXiv:2202.10204,
   (2022).

\bibitem{comon2009tensor}
{\sc P.~Comon, X.~Luciani, and A.~L. De~Almeida}, {\em Tensor decompositions,
  alternating least squares and other tales}, Journal of Chemometrics: A
  Journal of the Chemometrics Society, 23 (2009), pp.~393--405.

\bibitem{cox2013ideals}
{\sc D.~Cox, J.~Little, and D.~OShea}, {\em Ideals, varieties, and algorithms:
  an introduction to computational algebraic geometry and commutative algebra},
  Springer Science \& Business Media, 2013.

\bibitem{de2000multilinear}
{\sc L.~De~Lathauwer, B.~De~Moor, and J.~Vandewalle}, {\em A multilinear
  singular value decomposition}, SIAM journal on Matrix Analysis and
  Applications, 21 (2000), pp.~1253--1278.

\bibitem{de2017understanding}
{\sc C.~De~Sa, M.~Feldman, C.~R{\'e}, and K.~Olukotun}, {\em Understanding and
  optimizing asynchronous low-precision stochastic gradient descent}, in
  Proceedings of the 44th annual international symposium on computer
  architecture, 2017, pp.~561--574.

\bibitem{de2018high}
{\sc C.~De~Sa, M.~Leszczynski, J.~Zhang, A.~Marzoev, C.~R. Aberger,
  K.~Olukotun, and C.~R{\'e}}, {\em High-accuracy low-precision training},
  arXiv preprint arXiv:1803.03383,  (2018).

\bibitem{dolgov2015polynomial}
{\sc S.~Dolgov, B.~N. Khoromskij, A.~Litvinenko, and H.~G. Matthies}, {\em
  Polynomial chaos expansion of random coefficients and the solution of
  stochastic partial differential equations in the tensor train format},
  SIAM/ASA Journal on Uncertainty Quantification, 3 (2015), pp.~1109--1135.

\bibitem{domanov2017canonical}
{\sc I.~Domanov and L.~De~Lathauwer}, {\em Canonical polyadic decomposition of
  third-order tensors: Relaxed uniqueness conditions and algebraic algorithm},
  Linear Algebra and its Applications, 513 (2017), pp.~342--375.

\bibitem{dressler2021separability}
{\sc M.~Dressler, J.~Nie, and Z.~Yang}, {\em Separability of hermitian tensors
  and psd decompositions}, Linear and Multilinear Algebra,  (2021), pp.~1--28.

\bibitem{ge2017optimization}
{\sc R.~Ge and T.~Ma}, {\em On the optimization landscape of tensor
  decompositions}, Advances in Neural Information Processing Systems, 30
  (2017).

\bibitem{grasedyck2010hierarchical}
{\sc L.~Grasedyck}, {\em Hierarchical singular value decomposition of tensors},
  SIAM journal on matrix analysis and applications, 31 (2010), pp.~2029--2054.

\bibitem{guo2022learning}
{\sc B.~Guo, J.~Nie, and Z.~Yang}, {\em Learning diagonal gaussian mixture
  models and incomplete tensor decompositions}, Vietnam Journal of Mathematics,
  50 (2022), pp.~421--446.

\bibitem{haidar2020mixed}
{\sc A.~Haidar, H.~Bayraktar, S.~Tomov, J.~Dongarra, and N.~J. Higham}, {\em
  Mixed-precision iterative refinement using tensor cores on {GPUs} to
  accelerate solution of linear systems}, Proceedings of the Royal Society A,
  476 (2020), p.~20200110.

\bibitem{hawkins2022towards}
{\sc C.~Hawkins, X.~Liu, and Z.~Zhang}, {\em Towards compact neural networks
  via end-to-end training: A {Bayesian} tensor approach with automatic rank
  determination}, SIAM Journal on Mathematics of Data Science, 4 (2022),
  pp.~46--71.

\bibitem{hawkins2021bayesian}
{\sc C.~Hawkins and Z.~Zhang}, {\em Bayesian tensorized neural networks with
  automatic rank selection}, Neurocomputing, 453 (2021), pp.~172--180.

\bibitem{hettiarachchi2020integer}
{\sc D.~L.~N. Hettiarachchi, V.~S.~P. Davuluru, and E.~J. Balster}, {\em
  Integer vs. floating-point processing on modern fpga technology}, in 2020
  10th Annual Computing and Communication Workshop and Conference (CCWC), IEEE,
  2020, pp.~0606--0612.

\bibitem{horn2012matrix}
{\sc R.~A. Horn and C.~R. Johnson}, {\em Matrix analysis}, Cambridge university
  press, 2012.

\bibitem{hubara2017quantized}
{\sc I.~Hubara, M.~Courbariaux, D.~Soudry, R.~El-Yaniv, and Y.~Bengio}, {\em
  Quantized neural networks: Training neural networks with low precision
  weights and activations}, The Journal of Machine Learning Research, 18
  (2017), pp.~6869--6898.

\bibitem{huggins2019towards}
{\sc W.~Huggins, P.~Patil, B.~Mitchell, K.~B. Whaley, and E.~M. Stoudenmire},
  {\em Towards quantum machine learning with tensor networks}, Quantum Science
  and technology, 4 (2019), p.~024001.

\bibitem{jouppi2017datacenter}
{\sc N.~P. Jouppi, C.~Young, N.~Patil, D.~Patterson, G.~Agrawal, R.~Bajwa,
  S.~Bates, S.~Bhatia, N.~Boden, A.~Borchers, et~al.}, {\em In-datacenter
  performance analysis of a tensor processing unit}, in Proceedings of the 44th
  annual international symposium on computer architecture, 2017, pp.~1--12.

\bibitem{Kerr_CUTLASS_2022}
{\sc A.~Kerr, H.~Wu, M.~Gupta, D.~Blasig, P.~Ramini, D.~Merrill, A.~Shivam,
  P.~Majcher, P.~Springer, M.~Hohnerbach, J.~Wang, and M.~Nicely}, {\em
  {CUTLASS}}, 4 2022, \url{https://github.com/NVIDIA/cutlass}.

\bibitem{kim2015compression}
{\sc Y.-D. Kim, E.~Park, S.~Yoo, T.~Choi, L.~Yang, and D.~Shin}, {\em
  Compression of deep convolutional neural networks for fast and low power
  mobile applications}, arXiv preprint arXiv:1511.06530,  (2015).

\bibitem{knoll2020fastmri}
{\sc F.~Knoll, J.~Zbontar, A.~Sriram, M.~J. Muckley, M.~Bruno, A.~Defazio,
  M.~Parente, K.~J. Geras, J.~Katsnelson, H.~Chandarana, et~al.}, {\em
  {fastMRI}: A publicly available raw k-space and {DICOM} dataset of knee
  images for accelerated {MR} image reconstruction using machine learning},
  Radiology. Artificial intelligence, 2 (2020).

\bibitem{kolda2009tensor}
{\sc T.~G. Kolda and B.~W. Bader}, {\em Tensor decompositions and
  applications}, SIAM review, 51 (2009), pp.~455--500.

\bibitem{kolda2020stochastic}
{\sc T.~G. Kolda and D.~Hong}, {\em Stochastic gradients for large-scale tensor
  decomposition}, SIAM Journal on Mathematics of Data Science, 2 (2020),
  pp.~1066--1095.

\bibitem{landsberg2012tensors}
{\sc J.~M. Landsberg}, {\em Tensors: geometry and applications}, Representation
  theory, 381 (2012), p.~3.

\bibitem{nene1996columbia}
{\sc S.~A. Nene, S.~K. Nayar, H.~Murase, et~al.}, {\em Columbia object image
  library (coil-100)},  (1996).

\bibitem{nie2017generating}
{\sc J.~Nie}, {\em Generating polynomials and symmetric tensor decompositions},
  Foundations of Computational Mathematics, 17 (2017), pp.~423--465.

\bibitem{nie2020hermitian}
{\sc J.~Nie and Z.~Yang}, {\em Hermitian tensor decompositions}, SIAM Journal
  on Matrix Analysis and Applications, 41 (2020), pp.~1115--1144.

\bibitem{nie2018complete}
{\sc J.~Nie, Z.~Yang, and X.~Zhang}, {\em A complete semidefinite algorithm for
  detecting copositive matrices and tensors}, SIAM Journal on Optimization, 28
  (2018), pp.~2902--2921.

\bibitem{novikov2015tensorizing}
{\sc A.~Novikov, D.~Podoprikhin, A.~Osokin, and D.~P. Vetrov}, {\em Tensorizing
  neural networks}, Advances in neural information processing systems, 28
  (2015).

\bibitem{cupy_learningsys2017}
{\sc R.~Okuta, Y.~Unno, D.~Nishino, S.~Hido, and C.~Loomis}, {\em {CuPy}: A
  {NumPy}-compatible library for {NVIDIA GPU} calculations}, in Proc. Workshop
  on Machine Learning Systems in The Thirty-first Annual Conference on Neural
  Information Processing Systems (NIPS), 2017.

\bibitem{olivares2010accelerating}
{\sc R.~Olivares-Amaya, M.~Watson, R.~Edgar, L.~Vogt, Y.~Shao, and
  A.~Aspuru-Guzik}, {\em Accelerating correlated quantum chemistry calculations
  using graphical processing units and a mixed precision matrix multiplication
  library}, Journal of chemical theory and computation, 6 (2010), pp.~135--144.

\bibitem{orus2019tensor}
{\sc R.~Or{\'u}s}, {\em Tensor networks for complex quantum systems}, Nature
  Reviews Physics, 1 (2019), pp.~538--550.

\bibitem{oseledets2011tensor}
{\sc I.~V. Oseledets}, {\em Tensor-train decomposition}, SIAM Journal on
  Scientific Computing, 33 (2011), pp.~2295--2317.

\bibitem{richter2021solving}
{\sc L.~Richter, L.~Sallandt, and N.~N{\"u}sken}, {\em Solving high-dimensional
  parabolic {PDEs} using the tensor train format}, in International Conference
  on Machine Learning, PMLR, 2021, pp.~8998--9009.

\bibitem{shan2019exact}
{\sc J.~Shan, M.~R. Casu, J.~Cortadella, L.~Lavagno, and M.~T. Lazarescu}, {\em
  Exact and heuristic allocation of multi-kernel applications to multi-{FPGA}
  platforms}, in Proceedings of the 56th Annual Design Automation Conference
  2019, 2019, pp.~1--6.

\bibitem{sidiropoulos2017tensor}
{\sc N.~D. Sidiropoulos, L.~De~Lathauwer, X.~Fu, K.~Huang, E.~E. Papalexakis,
  and C.~Faloutsos}, {\em Tensor decomposition for signal processing and
  machine learning}, IEEE Transactions on Signal Processing, 65 (2017),
  pp.~3551--3582.

\bibitem{sun2020ultra}
{\sc X.~Sun, N.~Wang, C.-Y. Chen, J.~Ni, A.~Agrawal, X.~Cui, S.~Venkataramani,
  K.~El~Maghraoui, V.~V. Srinivasan, and K.~Gopalakrishnan}, {\em Ultra-low
  precision 4-bit training of deep neural networks}, Advances in Neural
  Information Processing Systems, 33 (2020), pp.~1796--1807.

\bibitem{vervliet2015randomized}
{\sc N.~Vervliet and L.~De~Lathauwer}, {\em A randomized block sampling
  approach to canonical polyadic decomposition of large-scale tensors}, IEEE
  Journal of Selected Topics in Signal Processing, 10 (2015), pp.~284--295.

\bibitem{yaman2019low}
{\sc B.~Yaman, S.~Weing{\"a}rtner, N.~Kargas, N.~D. Sidiropoulos, and
  M.~Ak{\c{c}}akaya}, {\em Low-rank tensor models for improved multidimensional
  {MRI}: Application to dynamic cardiac $ t\_1 $ mapping}, IEEE transactions on
  computational imaging, 6 (2019), pp.~194--207.

\bibitem{zbontar2018fastmri}
{\sc J.~Zbontar, F.~Knoll, A.~Sriram, T.~Murrell, Z.~Huang, M.~J. Muckley,
  A.~Defazio, R.~Stern, P.~Johnson, M.~Bruno, et~al.}, {\em {fastMRI}: An open
  dataset and benchmarks for accelerated {MRI}}, arXiv preprint
  arXiv:1811.08839,  (2018).

\bibitem{zhang2021fpga}
{\sc K.~Zhang, C.~Hawkins, X.~Zhang, C.~Hao, and Z.~Zhang}, {\em On-{FPGA}
  training with ultra memory reduction: A low-precision tensor method}, arXiv
  preprint arXiv:2104.03420,  (2021).

\bibitem{zhang2016big}
{\sc Z.~Zhang, T.-W. Weng, and L.~Daniel}, {\em Big-data tensor recovery for
  high-dimensional uncertainty quantification of process variations}, IEEE
  Transactions on Components, Packaging and Manufacturing Technology, 7 (2016),
  pp.~687--697.

\bibitem{zhang2014enabling}
{\sc Z.~Zhang, X.~Yang, I.~V. Oseledets, G.~E. Karniadakis, and L.~Daniel},
  {\em Enabling high-dimensional hierarchical uncertainty quantification by
  {ANOVA} and tensor-train decomposition}, IEEE Transactions on Computer-Aided
  Design of Integrated Circuits and Systems, 34 (2014), pp.~63--76.

\end{thebibliography}


\end{document}